\documentclass{amsart}
\usepackage{amssymb,amsmath,amsthm,amscd,float,color,mathtools, pdflscape}
\allowdisplaybreaks

\usepackage{color}

\usepackage{mathtools}

\usepackage[msc-links]{amsrefs}

\hypersetup{
  colorlinks   = true, 
  urlcolor     = blue, 
  linkcolor    = blue, 
  citecolor   = red 
}

\newcommand{\nocontentsline}[3]{}
\newcommand{\tocless}[2]{\bgroup\let\addcontentsline=\nocontentsline#1*{#2}\egroup}

\newcommand{\ch}[2]
{\begin{bmatrix}
 #1 \\
 #2\\
\end{bmatrix}}

\newcommand{\chr}[4]
{\begin{bmatrix}
 #1 & #2\\[3pt]
 #3 & #4
\end{bmatrix}}

\def\Z{\mathbb Z}
\def\Q{\mathbb Q}
\def\R{\mathbb R}
\def\C{\mathbb C}

\def\H{\mathcal H}
\def\M{{\mathcal M}}
\def\A{\mathcal A}

\def\m{\mathfrak m}

\def\det{\mbox{det }}

\def\J{\mbox{Jac }}

\def\<{\langle}
\def\>{\rangle}

\def\t{\tau}

\def\T{\theta}
\def\d{{\delta }}

\def\embd{\hookrightarrow}

\def\A{\mathcal A}
\def\T{\theta}

\newtheorem{thm}{Theorem}
\newtheorem{prop}[thm]{Proposition}
\newtheorem{lem}[thm]{Lemma}
\newtheorem{rem}[thm]{Remark}
\newtheorem{cor}[thm]{Corollary}

\theoremstyle{definition}

\theoremstyle{remark}

\newcommand{\im}{\textnormal{Im}\;\;\!\!\!}
\def\H{\mathbb{H}}
\newcommand{\supp}{\textnormal{supp}}

\newcommand{\sptwoz}{\operatorname{Sp}_4(\Z)}

\begin{document}
\title[Satake sextic in elliptic fibrations]{The Satake sextic in elliptic fibrations on K3}

\author{A. Malmendier}
\address{Department of Mathematics and Statistics, Utah State University,
Logan, UT 84322}
\email{andreas.malmendier@usu.edu}

\author{T. Shaska}
\address{Department of Mathematics and Statistics, Oakland University, Rochester, MI 48309}
\email{shaska@oakland.edu}
\begin{abstract}
We describe explicit formulas relevant to the F-theory/heterotic string duality  that reconstruct from a specific Jacobian elliptic fibration on the Shioda-Inose surface 
covering a generic Kummer surface the corresponding genus-two curve using the level-two Satake coordinate functions.
We derive explicitly the rational map on the moduli space of genus-two curves realizing the algebraic correspondence between a sextic curve and its Satake sextic.
We will prove that it is not the original sextic defining the genus-two curve, but its corresponding Satake sextic which is manifest in the
F-theory model, dual to the $\mathfrak{so}(32)$ heterotic string with an unbroken $\mathfrak{so}(28)\oplus \mathfrak{su}(2)$ gauge algebra.
\end{abstract}

\maketitle

\section{Introduction}

Constructing equations of algebraic curves from a given point in the moduli space or a given Jacobian has always been interesting to both mathematicians and physicists.  The only case where such constructions can be made explicit is the case of genus-two curves.  There have been attempts by other authors before  where equations of the genus-two curve is written in terms of the thetanulls of the Jacobian; see  \cite{psw} and  \cite{sw}.  

By a sextic curve we mean a projective curve of degree six.  
To each sextic curve one can associate another sextic curve, called the \emph{Satake sextic}.
The algebraic correspondence between these two sextics is quite complicated, and we give explicit formulas for its construction.
In fact, starting with a plane curve, for example in Rosenhain normal form, the computation of the Igusa invariants provides an effective method for computing the corresponding Satake sextic.
Conversely, starting with the roots of the Satake sextic we will derive explicit formulas for the reconstruction of the original sextic up to equivalence.
This will allow us to explicitly determine the rational map on the moduli space of genus-two curves realizing the correspondence between a sextic curve and its Satake sextic.

For a generic genus-two curve $\mathcal{C}$ the Jacobian variety $\operatorname{Jac}(\mathcal{C})$ is principally polarized abelian surface,
and the minimal resolution of the quotient by the involution automorphism is a special K3 surface. called the \emph{Kummer surface} $\mathrm{Kum}(\operatorname{Jac}\mathcal{C})$.
There is a closely related K3 surface, called the \emph{Shioda-Inose surface} $\mathrm{SI}(\operatorname{Jac}\mathcal{C})$, which carries
a \emph{Nikulin involution}, i.e., an automorphism of order two preserving the holomorphic two-form, such that quotienting by this involution and blowing up the
fixed points recovers the Kummer surface.
By using  the Shioda-Inose surface $\mathrm{SI}(\operatorname{Jac}\mathcal{C})$ that covers the Kummer surface,
one establishes a one-to-one correspondence between two different types of surfaces with the same Hodge-theoretic data, principally polarized abelian surfaces
and algebraic K3 surfaces polarized by a special lattice, which is known as \emph{geometric two-isogeny}.

In string theory, compactifications of the so-called type-IIB string in which the complex coupling varies over a base are generically referred to as F-theory.
The simplest such construction corresponds to a Jacobian elliptic fibration on a K3 surface. By taking this K3 surface to be the
Shioda-Inose surface $\mathrm{SI}(\operatorname{Jac}\mathcal{C})$ a phenomenon called \emph{F-theory/heterotic string duality} is manifested as 
the aforementioned geometric two-isogeny. An important question is whether the original genus-two curve $\mathcal{C}$ is still manifest in this F-theoretic description of 
non-geometric heterotic string backgrounds. We will prove that it is not the original sextic defining the genus-two curve $\mathcal{C}$, but the corresponding Satake sextic 
which is manifest in the F-theoretic data. In fact, the ramification locus of the Satake sextic is the genus-two component of the fixed point set of the Nikulin involution 
on $\mathrm{SI}(\operatorname{Jac}\mathcal{C})$.

This article is structured as follows: in Section~2 we give a brief review of principally polarized abelian surfaces, the 
thetanulls for genus two, and the Satake coordinate functions, as well as their relations to the Igusa invariants and Siegel modular forms. 
We then prove a Picard like result, which gives  the Rosenhein roots of a genus-two curve in terms of the thetanulls and also in terms of the Satake coordinate functions. 
These explicit formulas are instrumental in computing the rational map on the moduli space of genus-two curves that realizes the algebraic correspondence between the sextic and 
its corresponding Satake sextic.  In Section~3 we describe the construction of the Kummer surface $\mathrm{Kum}(\operatorname{Jac}\mathcal{C})$ and
Shioda-Inose surface $\mathrm{SI}(\operatorname{Jac}\mathcal{C})$, as well as the Jacobian elliptic fibrations on them which are relevant for the F-theory/heterotic string duality.  
We then prove that the positions of 7-branes with string charge $(1,0)$ in 
the F-theory model dual to the $\mathfrak{so}(32)$ heterotic string with an unbroken $\mathfrak{so}(28)\oplus \mathfrak{su}(2)$ gauge algebra and 
only one non-vanishing Wilson line form the ramification locus of the Satake sextic which is in algebraic correspondence with the genus-two curve $\mathcal{C}$.

\section{The correspondence between a sextic and its Satake sextic}
In this section we give a brief review of principally polarized abelian surfaces, the 
thetanulls, the Satake coordinate functions, and their relations to the Igusa invariants and Siegel modular forms. 
We then prove a Picard like result, which gives the Rosenhein roots of genus two curve in terms of the thetanulls and also in terms of the Satake coordinate functions. 
We also compute the rational map on the moduli space of genus-two curves realizing the correspondence between a sextic curve and its Satake sextic.

\subsection{Abelian surfaces}
The \emph{Siegel upper-half space} is the set of two-by-two symmetric matrices over $\mathbb{C}$ whose imaginary part is positive definite, i.e.,
\begin{equation*}
\label{Siegel_tau}
 \mathbb{H}_2 = \left. \left\lbrace \tau = \left( \begin{array}{cc} \tau_1 & z \\ z & \tau_2\end{array} \right) \right|
 \tau_1, \tau_2, z \in \mathbb{C}\,,\; \im{(\tau_1)} \, \im{(\tau_2}) > \im{(z)}^2\,, \; \im{(\tau_2)} > 0 \right\rbrace .
\end{equation*}
The Siegel three-fold is a quasi-projective variety of dimension three obtained from the Siegel upper half plane when
quotienting out by the action of the modular transformations $\Gamma_2:=
\sptwoz$, i.e., 
\begin{equation}
 \mathcal{A}_2 =  \mathbb{H}_2 / \Gamma_2 \;.
\end{equation} 
For each $\tau \in \mathbb{H}_2$ the columns of the matrix $\left[ \, \mathbb{I}_2  |  \tau \right]$ form a
lattice $\Lambda$ in  $\C^2$ and determine a principally polarized complex abelian surface $\mathbf{A}_{\tau} = \C^2/\Lambda$.
Two abelian surfaces $\mathbf{A}_{\tau}$  and $\mathbf{A}_{\tau'}$ are isomorphic if and only if there is a symplectic matrix $M \in \Gamma_2$
such that $\tau' = M (\tau)$. It follows that the Siegel three-fold $\mathcal{A}_2$ is also the set of isomorphism classes of principally polarized abelian surfaces.
The even Siegel modular forms of $\mathcal{A}_2$ are a polynomial ring in four free generators of degrees $4$, $6$, $10$ and $12$ that will be denoted by $\psi_4, \psi_6, \chi_{10}$ and $\chi_{12}$, respectively. Igusa showed in \cite{MR0229643} that for the full ring of modular forms, one needs an additional generator $\chi_{35}$ which is algebraically dependent on the others. 
We also define $\Gamma_2(2n) = \lbrace M \in \Gamma_2 | \, M \equiv \mathbb{I} \mod{2n}\rbrace$ 
with corresponding Siegel modular threefold $\mathcal{A}_2(2)$
such that $ \Gamma_2/\Gamma_2(2)\cong S_6$ where $S_6$ is the permutation group of order $720$.

If $\mathcal{C}$ is an irreducible nonsingular projective curve with self-intersection $\mathcal{C}\cdot \mathcal{C}=2$ then $\mathcal{C}$ is a smooth curve of genus two.
We choose a symplectic homology basis for $\mathcal{C}$, say $ \{ A_1, A_2, B_1, B_2 \},$ such that the intersection products 
$A_i \cdot A_j = B_i \cdot B_j =0$ and $A_i \cdot B_j= \d_{i j},$ where $\d_{i j}$ is the Kronecker delta. We choose a basis $\{ w_i\}$ for the
space of holomorphic one-forms such that $\int_{A_i} w_j = \d_{i j}$. The matrix 
\[\tau = \left[ \int_{B_i} w_j\right]\]
is  the \emph{period matrix} of $\mathcal{C}$ and $\J\!(\mathcal{C}) = \mathbf{A}_{\tau}$ is the Jacobian of $\mathcal{C}$. 
Moreover, the map $\jmath_\mathcal{C}: \mathcal{C} \to \operatorname{Jac}(\mathcal{C})$ is an embedding
of the moduli space of genus-two curves $\mathcal{M}_2$ into the space of  principally polarized abelian surfaces, i.e.,
\[ \M_2 \embd \A_2 \;,\] 
where the hermitian form associated to the divisor class 
$[\mathcal{C}]$ is the principal polarization $\rho$ on $\operatorname{Jac}(\mathcal{C})$. 
Moreover, a curve $\mathcal{C}$ of genus-two is called generic  if the N\'eron-Severi lattice is generated by $[\mathcal{C}]$, i.e., $\mathrm{NS}(\operatorname{Jac} \mathcal{C})=\mathbb{Z}[\mathcal{C}]$. 
Since we have $\rho^2=2$, the transcendental lattice is $T(\operatorname{Jac}\mathcal{C}) = H \oplus H \oplus \langle -2 \rangle$ in this case.\footnote{$H$ is the standard hyperbolic lattice 
with the quadratic form $q=x_1x_2$.}
Conversely, one can always regain $\mathcal{C}$ from the pair $(\operatorname{Jac}\mathcal{C}, \rho)$ where $\rho$ is a principal polarization. 

The Humbert surface $H_{\Delta}$ with invariant $\Delta$ is the space of principally polarized abelian surfaces admitting a symmetric endomorphism 
of discriminant $\Delta$. The discriminant $\Delta$ is always a positive integer $\equiv 0, 1\mod 4$. In fact, $H_{\Delta}$ is the image of the equation
\begin{equation}
 a \, \tau_1 + b \, z + c \, \tau_3 + d\, (z^2 -\tau_1 \, \tau_2) + e = 0 \;,
\end{equation}
with integers $a, b, c, d, e$ satisfying $\Delta=b^2-4\,a\,c-4\,d\,e$ and $\tau
= \bigl(\begin{smallmatrix}
\tau_1&z\\ z&\tau_2
\end{smallmatrix} \bigr) \in \mathbb{H}_2$.
Therefore, inside of $\mathcal{A}_2$ sit the Humbert surfaces $H_1$ and $H_4$ that are defined as the images of the rational divisors 
associated to $z=0$ and $\tau_1=\tau_2$, respectively. 
In fact, $H_1$ and $H_4$ form the two connected components of the singular locus of $\mathcal{A}_2$, and
their formal sum $H_1 + H_4$ is the vanishing divisor of $\chi_{35}$.

Furthermore, Torelli's theorem states that the map sending a curve $\mathcal{C}$ to its Jacobian variety $\mathrm{Jac}(C)$ induces a birational map from the moduli space $\mathcal{M}_2$ 
of genus-two curves to the complement of the Humbert surface $H_1$ in $\mathcal{A}_2$.
This locus is expressed in terms of modular forms as $\mathcal{A}_2 \backslash \, \supp{ (\chi_{10})}_0$.
That is, a period point $\tau$ is equivalent to a point with $z=0$, i.e., $\tau \in H_1$, if and only if $\chi_{10}(\tau)=0$,
if and only if the principally polarized abelian surface $\mathbf{A}_{\tau} $ is a product of two elliptic curves  $\mathbf{A}_{\tau} =\mathcal{E}_{\tau_1} \times \mathcal{E}_{\tau_2}$. 
In turn, the transcendental lattice is $T(\mathbf{A}_{\tau}) = H \oplus H$. 

On the other hand, it is known that the vanishing divisor of $Q= 2^{12} \, 3^9 \, \chi_{35}^2 /\chi_{10}$ is the 
Humbert surface $H_4$ \cite{MR1438983}, that is, a period point $\tau$ is equivalent to a point with $\tau_1=\tau_2$, i.e., $\tau \in H_4$,
if and only if $Q(\tau)=0$. In turn, the transcendental lattice degenerates to $T(\mathbf{A}_{\tau}) = H \oplus \langle 2 \rangle \oplus \langle -2 \rangle$. 
Bolza \cite{MR1505464} described the possible automorphism groups of genus-two curves defined by sextics. In particular, 
he proved that a sextic curve $Y^2=F(X)$ defining the genus-two curve $\mathcal{C}$ with $\mathbf{A}_{\tau}=\mathrm{Jac}(C)$
has an extra involution, which then can be represented as $(X,Y) \mapsto (-X,Y)$, if and only if $Q(\tau)=0$.

\subsection{Thetanulls for genus two}
For any $z \in \C^2$ and $\t \in \H_2$ \emph{Riemann's theta function} is defined as
\[ \theta (z , \t) = \sum_{u\in \Z^2} e^{\pi i ( u^t \t u + 2 u^t z )  } \]
where $u$ and $z$ are two-dimensional column vectors and the products involved in the formula are matrix
products. The fact that the imaginary part of $\t$ is positive makes the series absolutely convergent over
any compact sets. Therefore, the function is analytic.  The theta function is holomorphic on $\C^2\times
\H_2$ and satisfies
\[\theta(z+u,\tau)=\theta(z,\tau),\quad \theta(z+u\tau,\tau)=e^{-\pi i( u^t \tau u+2z^t u )}\cdot
\theta(z,\tau),\]
where $u\in \Z^2$; see \cite{MR2352717}   for details. Any point $e \in \operatorname{Jac}(\mathcal{C})$ can be written uniquely as
$e = (b,a) \begin{pmatrix} \mathbb{I}_2 \\ \tau \end{pmatrix}$, where $a,b \in \R^2$.  We shall use the notation $[e] =
\ch{a}{b}$ for the characteristic of $e$. For any $a, b \in \Q^2$, the theta function with rational
characteristics is defined as
\[ \theta  \ch{a}{b} (z , \t) = \sum_{u\in \Z^2} e^{\pi i ( (u+a)^t \t (u+a) + 2 (u+a)^t (z+b) )  }. \]
When the entries of column vectors $a$ and $b$ are from the set $\{ 0,\frac{1}2\}$, then the
characteristics $ \ch {a}{b} $ are called the \emph{half-integer characteristics}. The corresponding theta
functions with rational characteristics are called \emph{theta characteristics}. A scalar obtained by
evaluating a theta characteristic at $z=0$ is called a \emph{theta constant}. Points of order $n$ on $\operatorname{Jac}(\mathcal{C})$
are called the $\frac 1 n$-\emph{periods}. Any half-integer characteristic is given by
\[
\m = \frac{1}2m = \frac{1}2
\begin{pmatrix} m_1 & m_2  \\ m_1^{\prime} & m_2^{\prime} \end{pmatrix} \;,
\]
where $m_i, m_i^{\prime} \in \Z.$  For $\gamma = \ch{\gamma ^\prime}{\gamma^{\prime \prime}} \in
\frac{1}2\Z^4/\Z^4$ we define $e_*(\gamma) = (-1)^{4 (\gamma ^\prime)^t \gamma^{\prime \prime}}.$
Then,
\[ \theta [\gamma] (-z , \t) = e_* (\gamma) \theta [\gamma] (z , \t) \;.\]
We say that $\gamma$ is an  \textbf{even}  (resp.  \textbf{odd}) characteristic if $e_*(\gamma) = 1$ (resp.
$e_*(\gamma) = -1$). 

For any genus-two curve we have six odd theta characteristics and ten even theta characteristics;  see \cite{psw}   for details. 
The following are the sixteen theta characteristics, where the first ten are even and the last six are odd. 
We denote the even theta constants by 
\begin{small}
\[
\begin{split}
\theta_1 = \chr {0}{0}{0}{0} ,  \,  
\theta_2 = \chr {0}{0}{\frac{1}2} {\frac{1}2} ,  \,
\theta_3 &= \chr {0}{0}{\frac{1}2}{0} , \, 
\theta_4 = \chr {0}{0}{0}{\frac{1}2} , \,   \;
\theta_5 = \chr{\frac{1}2}{0}{0}{0} ,\\[5pt]
  \theta_6  = \chr {\frac{1}2}{0}{0}{\frac{1}2} , \, 
  \theta_7 = \chr{0}{\frac{1}2} {0}{0} , \, 
  \theta_8 &= \chr{\frac{1}2}{\frac{1}2} {0}{0} , \, 
  \theta_9 = \chr{0}{\frac{1}2} {\frac{1}2}{0} , \, 
  \theta_{10} = \chr{\frac{1}2}{\frac{1}2} {\frac{1}2}{\frac{1}2} ,
\end{split}
\]
\end{small}
where we write
\begin{equation}
\label{Eqn:theta_short}
 \theta_i(z) \quad \text{instead of} \quad \theta  \ch{a^{(i)}}{b^{(i)}} (z , \tau) \quad \text{with $i=1,\dots ,10$,}
\end{equation}
and $\theta_i =\theta_i(0)$. 
Similarly, the odd theta functions correspond to the following characteristics
\[   \chr{0}{\frac{1}2} {0}{\frac{1}2} , \,
   \chr{0}{\frac{1}2} {\frac{1}2}{\frac{1}2} , \,
    \chr{\frac{1}2}{0} {\frac{1}2}{0} , \, \,
    \chr{\frac{1}2}{\frac{1}2} {\frac{1}2}{0} , \,
    \chr{\frac{1}2}{0} {\frac{1}2}{\frac{1}2} , \,
    \chr{\frac{1}2}{\frac{1}2} {0}{\frac{1}2}  \;.\]
Thetanulls are modular forms of $\mathcal{A}_2(2)$,  and the even theta fourth powers define a compactification of $\mathcal{A}_2(2)$ by $\operatorname{Proj}[\T^4_1: \dots : \T^4_{10}]$,
known as the \emph{Satake compactfication}. $\theta_1, \dots , \theta_4$ are called \textit{fundamental thetanulls}; see \cite{psw} for details. They are determined via the G\"opel systems. 
We have the following Frobenius identities relating the remaining theta constants to the fundamental thetanulls
\begin{equation}
\label{Eq:FrobeniusIdentities}
\begin{array}{lllclll}
 \theta_5^2 \theta_6^2 & = & \theta_1^2 \theta_4^2 - \theta_2^2 \theta_3^2 \,, &\qquad
 \theta_5^4 + \theta_6^4 & =& \theta_1^4 - \theta_2^4 - \theta_3^4 + \theta_4^4 \,, \\[0.2em]
 \theta_7^2 \theta_9^2 & = & \theta_1^2 \theta_3^2 - \theta_2^2 \theta_4^2 \,, &\qquad
 \theta_7^4 + \theta_9^4 &= & \theta_1^4 - \theta_2^4 + \theta_3^4 - \theta_4^4 \, , \\[0.2em]
 \theta_8^2 \theta_{10}^2 & = & \theta_1^2 \theta_2^2 - \theta_3^2 \theta_4^2 \, , &\qquad
 \theta_8^4 + \theta_{10}^4 & = & \theta_1^4 + \theta_2^4 - \theta_3^4 - \theta_4^4 \,, 
\end{array}
\end{equation}
as well as the following mixed relations
\begin{equation}
\label{Eq:FrobeniusIdentitiesb}
\begin{array}{lllclll}
 \theta_5^2 \theta_9^2 & = & \theta_3^2 \theta_8^2 - \theta_4^2 \theta_{10}^2 \,,&\qquad
 \theta_5^2 \theta_7^2 & = & \theta_1^2 \theta_8^2 - \theta_2^2 \theta_{10}^2 .
\end{array}
\end{equation}

Let a genus-two curve $\mathcal{C}$ be given by 
\begin{equation} \label{Rosen2}
Y^2= F(X) = X(X-1)(X-\lambda_1)(X-\lambda_2)(X-\lambda_3).
\end{equation} 
The ordered tuple $(\lambda_1, \lambda_2, \lambda_3)$ where the $\lambda_i$ are all distinct and different from $0, 1, \infty$
determines a point in $\mathcal{M}_2(2)$, the moduli space of genus-two curves together with a level-two structure, 
and, in turn, a level-two structure on the corresponding Jacobian variety, i.e., a point in the moduli space of principally polarized abelian surfaces
with level-two structure $\mathcal{A}_2(2)$.  As functions on $\mathcal{M}_2(2)$, the Rosenhain invariants generate its coordinate ring $\mathbb{C}(\lambda_1, \lambda_2, \lambda_3)$ and 
hence generate the function field of $\mathcal{A}_2(2)$.

The three $\lambda$-parameters in the Rosenhain normal~(\ref{Rosen2}) 
can be expressed as ratios of even theta constants by Picard's lemma. There are 720 choices for such expressions since
the forgetful map $\mathcal{M}_2(2) \to \mathcal{M}_2$ is a Galois covering of degree $720 = |S_6|$ since $S_6$ acts on the 
roots of $F$ by permutations. Any of the $720$ choices may be used, we picked the one from \cite{MR2367218}:
%
%
\begin{lem} If $\mathcal{C}$ is a genus-two curve with period matrix $\tau$ and $\chi_{10}(\tau)\not =0$, then $\mathcal{C}$  is equivalent to the 
curve~(\ref{Rosen2}) with Rosenhain parameters $\lambda_1, \lambda_2, \lambda_3$  given by \begin{equation}\label{Picard}
\lambda_1 = \frac{\theta_1^2\theta_3^2}{\theta_2^2\theta_4^2} \,, \quad \lambda_2 = \frac{\theta_3^2\theta_8^2}{\theta_4^2\theta_{10}^2}\,, \quad \lambda_3 =
\frac{\theta_1^2\theta_8^2}{\theta_2^2\theta_{10}^2}\,.
\end{equation}
Conversely, given three distinct complex numbers $(\lambda_1, \lambda_2, \lambda_3)$ different from $0, 1, \infty$ there is 
complex abelian surface $\mathbf{A}_{\tau}$ with period matrix $[\mathbb{I}_2 | \tau]$
such that $\mathbf{A}_{\tau}=\operatorname{Jac} (\mathcal{C})$ where $\mathcal{C}$ is the genus-two curve with period matrix $\tau$.
\end{lem}

\subsection{Igusa functions and Siegel modular forms} 
Let $I_2, \dots , I_{10}$ denote Igusa invariants of the binary sextic $Y^2=F(X)$ as defined in \cite{SV}*{Eq.~\!9} and explicitly given 
by Equations~(\ref{IgRos}) in the appendix for a curve in Rosenhain normal form~(\ref{Rosen2}).  The Igusa functions or absolute invariants are defined as 
\[  (j_1, j_2, j_3 ) = \left( \frac {I_2^5} {I_{10}}, \frac {I_4 I_2^3} {I_{10}}, \frac {I_6 I_2^2} {I_{10}}  \right) \;. \]
Two genus-two curves $\mathcal{C}$ and $\mathcal{C}^\prime$ are isomorphic if and only if 
\[ \left( j_1, j_2, j_3 \right) = \left( j_1^\prime, j_2^\prime, j_3^\prime \right)\;. \] 
Moreover, $j_1, j_2, j_3$ are given as rational functions of fourth powers of the fundamental theta functions $\theta_1, \dots , \theta_4$. 

The even Siegel modular forms of $\mathcal{A}_2$ form a polynomial ring in four free generators of degrees $4$, $6$, $10$ and $12$ 
denoted by $\psi_4, \psi_6, \chi_{10}$ and $\chi_{12}$, respectively.
Igusa \cite{MR0229643}*{p.~\!848} proved that the relation between the Igusa invariants of the binary sextic $Y^2=F(X)$ defining a genus-two
curve $\mathcal{C}$ with period matrix $\tau$ and the even Siegel modular forms are as follows:
\begin{equation}
\label{invariants}
\begin{split}
 I_2(F) & = -2^3 \cdot 3 \, \dfrac{\chi_{12}(\tau)}{\chi_{10}(\tau)} \;, \\
 I_4(F) & = \phantom{-} 2^2 \, \psi_4(\tau) \;,\\
 I_6(F) & = -\frac{2^3}3 \, \psi_6(\tau) - 2^5 \,  \dfrac{\psi_4(\tau) \, \chi_{12}(\tau)}{\chi_{10}(\tau)} \;,\\
 I_{10}(F) & = -2^{14} \, \chi_{10}(\tau) \;.
\end{split}
\end{equation}
Notice that the Igusa invariant $I_{10}$ is the discriminant of the sextic $Y^2=F(X)$, i.e., $\Delta_F =   I_{10}(F)$.
Conversely, for $r\not = 0$ the point $[I_2 : I_4 : I_6 : I_{10}]$ in weighted projective space equals
\begin{small}
\begin{equation*}
\label{IgusaClebschProjective}
\begin{split}
 \Big[ 2^3 \, 3 \, (3r\chi_{12})\, : \, 2^2 3^2 \, \psi_4 \,  (r\chi_{10})^2 \, : \, 2^3\, 3^2\, \Big(4 \psi_4  \, (3r\chi_{12})+ \psi_6 \,(r\chi_{10}) \Big)\, (r\chi_{10})^2: 2^2 \,  (r\chi_{10})^6 \Big].
 \end{split}
\end{equation*}
\end{small}
Furthermore, Igusa showed in \cite{MR0229643} that for the full
ring of modular forms of $\mathcal{A}_2$, one needs an additional generator $\chi_{35}$ which is algebraically dependent on the others. In fact, its square can be written as follows:
\begin{small}
\begin{equation}\label{chi35sqr}
\begin{split}
\chi_{35}^2 & = \frac{1}{2^{12} \, 3^9} \; \chi_{10} \,  \Big(  
2^{24} \, 3^{15} \; \chi_{12}^5 - 2^{13} \, 3^9 \; \psi_4^3 \, \chi_{12}^4 - 2^{13} \, 3^9\; \psi_6^2 \, \chi_{12}^4 + 3^3 \; \psi_4^6 \, \chi_{12}^3 \\
& - 2\cdot 3^3 \; \psi_4^3 \, \psi_6^2 \, \chi_{12}^3 - 2^{14}\, 3^8 \; \psi_4^2 \, \psi_6 \, \chi_{10} \, \chi_{12}^3 -2^{23}\, 3^{12} \, 5^2\, \psi_4 \, \chi_{10}^2 \, \chi_{12}^3  + 3^3 \, \psi_6^4 \, \chi_{12}^3\\
& + 2^{11}\,3^6\,37\,\psi_4^4\,\chi_{10}^2\,\chi_{12}^2+2^{11}\,3^6\,5\cdot 7 \, \psi_4 \, \psi_6^2\, \chi_{10}^2 \, \chi_{12}^2 -2^{23}\, 3^9 \, 5^3 \, \psi_6\, \chi_{10}^3 \, \chi_{12}^2 \\
& - 3^2 \, \psi_4^7 \, \chi_{10}^2 \, \chi_{12} + 2 \cdot 3^2 \, \psi_4^4 \, \psi_6^2 \, \chi_{10}^2 \, \chi_{12} + 2^{11} \, 3^5 \, 5 \cdot 19 \, \psi_4^3 \, \psi_6 \, \chi_{10}^3 \, \chi_{12} \\
&  + 2^{20} \, 3^8 \, 5^3 \, 11 \, \psi_4^2 \, \chi_{10}^4 \, \chi_{12} - 3^2 \, \psi_4 \, \psi_6^4 \, \chi_{10}^2 \, \chi_{12} + 2^{11} \, 3^5 \, 5^2 \, \psi_6^3 \, \chi_{10}^3 \, \chi_{12}  - 2 \, \psi_4^6 \, \psi_6 \, \chi_{10}^3 \\
 & - 2^{12} \, 3^4 \, \psi_4^5 \, \chi_{10}^4 + 2^2 \, \psi_4^3 \, \psi_6^3 \, \chi_{10}^3 + 2^{12} \, 3^4 \, 5^2 \, \psi_4^2 \, \psi_6^2 \, \chi_{10}^4 + 2^{21} \, 3^7 \, 5^4 \, \psi_4 \, \psi_6 \, \chi_{10}^5 \\
 & - 2 \, \psi_6^5 \, \chi_{10}^3 + 2^{32} \, 3^9 \, 5^5 \, \chi_{10}^6 \Big) \;.
\end{split}
\end{equation}
\end{small}
Hence, $Q= 2^{12} \, 3^9 \, \chi_{35}^2 /\chi_{10}$ is a polynomial of degree $60$
in the even generators.


\subsection{The Satake coordinate functions}

For a symplectic matrix $T \in \mathrm{Sp}_4(\Z)$, there is an induced action on the characteristics of the theta constants 
$\m \mapsto T \cdot \m$ such that the characteristic  $T \cdot \m$ has the same parity as $\m$ and $T \cdot \m =\m$ if $T \equiv \mathbb{I} (2)$.
The latter property implies that $\Gamma_2 /\Gamma_2(2) \cong \operatorname{Sp}_4(\mathbb{F}_2)$ acts on the characteristics.
It turns out that this action is transitive on the six odd characteristics and gives an isomorphism between the permutation group $S_6$
and $\operatorname{Sp}_4(\mathbb{F}_2)$ \cite{MR2744034}.

On any function $f: \mathbb{H}_2 \to \mathbb{C}$, a right action of $T =  \bigl(\begin{smallmatrix}
a&b\\ c&d
\end{smallmatrix} \bigr)  \in \mathrm{Sp}_4(\R)$ is given by setting $f\circ[T](\tau):=\det(c\tau+d)^{-2} \, f(T\cdot \tau)$.
It then follows that  $\T^4_\m\circ[T^{-1}]=\pm \T^4_{T \cdot \m}$ for all $T \in \mathrm{Sp}_4(\Z)$ with $\Gamma_2(2)$ acting trivially.
Thus, $S_6$ acts on the vector space $\mbox{Mat}_2 \left( \Gamma (2) \right)$ spanned by the ten even theta fourth powers.
The vector space $\mbox{Mat}_2 \left( \Gamma (2) \right)$ is a five-dimensional vector space, and we will use the set of theta functions
\[ \left\{ \T_1^4,  \T_2^4,  \T_3^4 , \T_4^4, \T_5^4 \right\} \]
as a basis for the space $\mbox{Mat}_2 \left( \Gamma_2(2) \right)$.  In fact, the other
even theta fourth powers are represented in terms of this basis as
\begin{equation}
\label{theta_reduction}
\begin{split}
 \T_6^4 & = \T_1^4 - \T_2^4 - \T_3^4 + \T_4^4 - \T_5^4, \\
 \T_7^4 & = \T_3^4 - \T_4^4 + \T_5^4 ,\\
 \T_8^4 & = \T_2^4 - \T_4^4 + \T_5^4 ,\\
 \T_9^4 & = \T_1^4 - \T_2^4 - \T_5^4 ,\\
 \T_{10}^4 & = \T_1^4 - \T_3^4 - \T_5^4 .
\end{split}
\end{equation}
If we set $u_k = \sum_\m \T_\m^{4k}$ it can be checked using the Frobenius identities~(\ref{Eq:FrobeniusIdentities}) 
that $u_2^2 = 4u_4$; see \cite{MR2744034}. Therefore, Equations~(\ref{theta_reduction}) realize the Satake compactification
of $\mathcal{A}_2(2)$ as the quartic threefold $u_2^2 = 4u_4$ 
in $\operatorname{Proj}[\T^4_1: \dots : \T^4_5]$.

The following functions are linear combinations of the fourth-powers of even generators
and are called \emph{level-two Satake coordinate functions} 
\begin{equation}
\label{Satake2Theta}
\begin{split}
x_1 & = -\T_1^4 + 2 \, \T_2^4 +2 \, \T_3^4 - \T_4^4+ 3 \, \T_5^4,\\
x_2 & = -\T_1^4 + 2 \, \T_2^4 - \, \T_3^4 - \T_4^4, \\
x_3 & = -\T_1^4-\T_2^4 -\T_3^4+ 2\, \T_4^4,\\
x_4 & = 2 \, \T_1^4-\T_2^4 -\T_3^4- \T_4^4, \\
x_5 & = - \T_1^4-\T_2^4 + 2\,\T_3^4- \T_4^4, \\
x_6 & = 2 \, \T_1^4 - \T_2^4 - \T_3^4 + 2 \, \T_4^4- 3 \, \T_5^4.
\end{split}
\end{equation}
It is obvious that $\sum_i x_i =0$.
A direct computation also shows that the group $S_6$ acts on $(x_1, \dots, x_6)$ by permutation \cite{MR2744034}.

We have the following lemma:
\begin{lem}
For $\tau \in \mathcal{A}_2(2)$ with $\chi_{10}(\tau) \not =0$ the level-two Satake coordinate functions $x_1, \dots x_6$ determine 
a curve in Rosenhain normal from~(\ref{Rosen2}) with $\Delta_F =  -2^{14} \, \chi_{10}(\tau)$ and Igusa invariants~(\ref{invariants})
where the Rosenhain roots $(\lambda_1, \lambda_2, \lambda_3)$ are given
by relations
\begin{equation}\label{Picard2}
\lambda_1 = \frac{1}2 + \frac{\T_1^4\T_3^4-\T_7^4\T_9^4}{2 \, \T_2^4\T_4^4}, \quad 
\lambda_2 = \frac{1}2 +  \frac{\T_3^4\T_8^4-\T_5^4\T_9^4}{2\, \T_4^4\T_{10}^4},\quad 
\lambda_3 =  \frac{1}2 + \frac{\T_1^4\T_8^4-\T_5^4\T_7^4}{2 \, \T_2^4\T_{10}^4},
\end{equation}
and
\begin{equation}
\label{Theta2Satake}
\begin{array}{lllclll}
 \T_1^4 & = & - \frac{1}3 \, \left( x_2 + x_3 + x_5\right), &\qquad
 \T_2^4 & = & - \frac{1}3 \, \left( x_3 + x_4 + x_5\right) , \\[0.2em]
 \T_3^4 & = & - \frac{1}3 \, \left( x_2 + x_3 + x_4\right) , &\qquad
 \T_4^4 & = & - \frac{1}3 \, \left( x_2 + x_4 + x_5\right) , \\[0.2em]
 \T_5^4 & = & \phantom{-} \frac{1}3 \, \left( x_1 + x_3 + x_4\right) , &\qquad
 \T_6^4 & = & - \frac{1}3 \, \left( x_1 + x_2 + x_5\right) , \\[0.2em]
 \T_7^4 & = & \phantom{-}  \frac{1}3 \, \left( x_1 + x_4 + x_5\right) , &\qquad
 \T_8^4 & = & \phantom{-}  \frac{1}3 \, \left( x_1 + x_2 + x_4\right) , \\[0.2em]
 \T_9^4 & = &  - \frac{1}3 \, \left( x_1 + x_2 + x_3\right) , &\qquad
 \T_{10}^4 & =  & - \frac{1}3 \, \left( x_1 + x_3 + x_5\right) .
\end{array}
\end{equation}
\end{lem}
\begin{proof}
The proof follows when using the Frobenius identities~(\ref{Eq:FrobeniusIdentities}) to 
re-write the Rosenhain roots in terms of fourth powers of theta functions and solving
Equations~(\ref{Satake2Theta}) for $\T_1^4, \dots, \T_{10}^4$.
\end{proof}

Define the $j$-th power sums $s_j$ are defined by 
\[ s_j = \sum_{i=1}^6 x_i^j \;. \]
Apart from the obvious identity $s_1=0$, the relation $u_2^2 = 4 \, u_4$ implies the only other relation $s_2^2 = 4 \, s_4$.
Therefore, Equations~(\ref{Satake2Theta}) define an embedding of the Satake compactfication
of $\mathcal{A}_2(2)$, into $\mathbb{P}^5 \ni [x_1:x_2:x_3:x_4:x_5:x_6]$. The image 
in $\mathbb{P}^5$, known as the \emph{Igusa quartic}, is the intersection of the hyperplane $s_1=0$ and the quartic hypersurface $s_2^2 = 4 \, s_4$.

We have the following lemma describing the descent to the Igusa invariants:
 \begin{lem}
 We have the following relations between $s_2, s_3, s_5, s_6$ and the Igusa invariants
 \begin{equation}
 \label{Satake2Igusa}
 \begin{split}
  s_2 & = 3 \, I_4, \\
   s_3 & = \frac32 I_2 I_4 - \frac{9}2 I_6, \\
   s_5 & = \frac{15}{8} I_2 I_4^2 - \frac{45}{8} I_4 I_6 + 1215 \, I_{10}, \\
  s_6 & = \frac{27}{16} I_4^3 + \frac3{8} I_2^2 I_4^2 - \frac{9}4 I_2 I_4 I_6 + \frac{27}{8} I_6^2 + \frac{729}4 I_2 I_{10} \;,\qquad
 \end{split}
 \end{equation}
 and, conversely,
 \begin{equation}
 \begin{split}
  I_2 & = \frac53 \frac{3 \, s_2^3 + 8 s_3^2 - 48 s_6}{5 s_2 s_3 - 12 s_5}, \\
   I_4 & = \frac{1}3 s_2, \\
  I_6 & = \frac{1}{27} \frac{15 \, s_2^4 + 10 s_2 s_3^2-240 s_2 s_6+ 72 s_3 s_5}{5 \, s_2 s_3 - 12 s_5},\\
  I_{10} & = - \frac{1}{2916} s_2 s_3 + \frac{1}{1215} s_5 .
 \end{split}
 \end{equation}
 \end{lem}
 \begin{proof}
 Using the definition of the Igusa invariants we prove the lemma by explicit computation.
\end{proof}

\subsection{The Satake sextic}
We combine the level-two Satake functions in another plane sextic curve, called the \emph{Satake sextic}, given by
\[
 f(x) = \prod_{i=1}^6 (x -x_i)
\]
The coefficients of the Satake sextic are polynomials in  $\Z \left[ \frac 1 2 , \frac 1 3 , s_2, s_3, s_4, s_6 \right]$.
In fact, we obtain
\[
 f(x) = x^6 + \sum_{i=1}^6 \frac{(-1)^i}{i!} \mathcal{B}_i(Z) \, x^{6-i} \,
 \]
 where $Z= \lbrace s_1, -s_2, 2! \, s_3, - 3!\, s_4, 4! \, s_5, -5! \, s_6 \rbrace$ and $\mathcal{B}_i(Z)$ is the 
 complete Bell polynomial of order $i$ in the variables contained in $Z$.
The following proposition follows:
\begin{prop}\label{Satake6}
For $\tau \in \mathcal{A}_2$ the level-two Satake coordinate functions $x_1, \dots , x_6$ are the roots of the Satake 
polynomial $f \in \Z \left[ \psi_4, \psi_6, \chi_{10}, \chi_{12}  \right] [x] $  given by
\begin{equation}
\label{SatakeSextic}
 f(x) =  \, \left(x^3- \frac{s_2}{4} \, x - \frac{s_3}{6}\right)^2 + \left( \frac{s_2 s_3}{12} - \frac{s_5}{5} \right) x + \frac{s_2^3}{96} + \frac{s_3^2}{36} - \frac{s_6}{6} ,
 \end{equation}
  or, equivalently,
 \begin{equation}
\label{SatakeSextic_3}
\begin{split}
 f(x) = & \, \left(x^3 - 3 \, \psi_4  \, x - 2 \, \psi_6 \right)^2 + 2^{14}3^5 \, \left( \chi_{10} \, x -3 \chi_{12} \right) .
 \end{split}
 \end{equation}
A genus-two curve  $\mathcal{S}$ is defined by $y^2=f(x)$ iff the discriminant does not vanish, i.e.,
$$
 \Delta_f = I_{10}(f) = 2^{52} 3^{21} \, Q(\tau) \not = 0 \;.
$$
\end{prop}

\begin{proof}
The proof follows from explicit computation of the Bell polynomials and using the relations $s_1=0$ and $s_2^2 = 4 \, s_4$.
One then checks that the discriminant of the polynomial in Equation~(\ref{SatakeSextic}) after using relations~(\ref{Satake2Igusa}) and (\ref{invariants})
coincides up to factor with $Q= 2^{12} \, 3^9 \, \chi_{35}^2 /\chi_{10}$ where $\chi_{35}^2$ was given in Equation~(\ref{chi35sqr}).
\end{proof}


A genus-two curve $\mathcal{C}$ with period matrix $\tau$ determines a Satake sextic polynomial (\ref{SatakeSextic_3}) in Proposition~\!\ref{Satake6}.
Therefore, we get a map $\Phi:  \mathcal C  \mapsto \mathcal S$ by mapping the genus-two curve $\mathcal{C}$ to the Satake sextic 
$y^2=f(x)$ with $f$ given in Equation~(\ref{SatakeSextic}).  Note that this map is defined as a map on the moduli space $\M_2$ if and only if $\chi_{10}(\tau)\not=0$ and 
$\Delta_f= - 2^{14} \chi_{10}(\tau') = 2^{52} 3^{21} Q(\tau)\not=0$ because only then do we have genus-two curves in the domain and range with period matrixes $\tau$ and $\tau'$, respectively.
Though not by explicit formulas it was proved in~\cite{MR2069800} that this map is a rational map of degree 16.  
The following lemma provides the explicit formulas:
\begin{prop}
Let $(j_1, j_2, j_3)$ and $(j_1^\prime, j_2^\prime, j_3^\prime)$ be the absolute invariants of $\mathcal{C}$ and $\mathcal{S}$, respectively.
The map $\Phi: \M_2 \backslash \supp{ (\chi_{35})}_0 \to \M_2 $ with $\mathcal C  \mapsto \mathcal S$ is given by
\begin{equation*}
\label{Transfo}
\begin{split}
 j_1^\prime 	= \frac{64}{729} 	\, \frac{g^{(1)}(j_1,j_2,j_3)}{q(j_1,j_2,j_3)},\quad
 j_2^\prime 	= \frac{4}{729} 	\, \frac{g^{(2)}(j_1,j_2,j_3)}{q(j_1,j_2,j_3)}, \quad
 j_3^\prime 	= \frac{1}{729} 	\, \frac{g^{(3)}(j_1,j_2,j_3)}{q(j_1,j_2,j_3)},
\end{split}
\end{equation*} 
where the polynomials $g^{(n)}$ for $n=1,2,3$ with integer coefficients are given in Equations~(\ref{Components_Phi}) in the appendix,
and the denominator is related to the modular form $Q= 2^{12} \, 3^9 \, \chi_{35}^2 /\chi_{10}$
by the equation 
$$
  Q(\tau) \, I_2^{-30} = 2^{-63} \, j_1^{-15} \, q(j_1,j_2,j_3).
$$
\end{prop}

\proof
Let $\mathcal{C}$ be a genus-two curve with period matrixes $\tau$.
Similarly, let $\mathcal{S}$ be a genus-two curve with period matrixes $\tau'$ defined by the sextic $y^2=f(x)$ with $f$ given in Equation~(\ref{SatakeSextic_3}).
We compute the Igusa invariants $I_n(f)$ for $n=2, 3, 6, 10$ and use Equations~(\ref{invariants}) to compute $\psi_4(\tau'), \psi_6(\tau'), \chi_{10}(\tau')$ and $\chi_{12}(\tau')$.
It follows that
\begin{equation}
\begin{split}
  \psi_4(\tau') & = \phantom{-}  2^4 3^6 \, K(\tau) ,\\
  \psi_6(\tau') & = -  2^6 3^9 \, L(\tau) ,\\
  \chi_{10}(\tau') & = - 2^{38} 3^{21} \, Q(\tau) ,\\
  \chi_{12}(\tau') & = \phantom{-} 2^{40} 3^{23} \, Q(\tau) \, M(\tau),
\end{split}
\end{equation}
where $K(\tau) = \psi_4(\tau)^6 + \dots$ and $L(\tau) = \psi_6(\tau)^6 + \dots$ are polynomials in the four even generators
of degree $24$ and $36$, respectively, 
$Q(\tau) = 2^{32} 3^9 5^5 \chi_{10}(\tau)^6 + \dots $
is a polynomial of degree $60$, and
$$
M(\tau) = \psi_4(\tau)^3 - \psi_6(\tau)^2 + 2^{13} \, 3^4 \, 5 \, \chi_{12}(\tau) \;.
$$
We also find
$$
 Q(\tau') = 2^{210} 3^{132} \, Q(\tau)^3 \, N(\tau)^2 \;,
$$ 
where $N(\tau)=\psi_4(\tau)^{15} \chi_{10}(\tau)^3 + \dots$ is a polynomial of degree $90$.
Then, the Igusa functions $(j_1^\prime, j_2^\prime, j_3^\prime)$ of the genus-two curve $\mathcal{S}$ are given by
\begin{equation}
\begin{split}
 j_1^\prime 	&= \frac{64}{729} \, \frac{m(j_1,j_2,j_3)^5}{q(j_1,j_2,j_3)},\\
 j_2^\prime 	&= \frac{4}{729} 	\, \frac{m(j_1,j_2,j_3)^3 \, k(j_1,j_2,j_3)}{q(j_1,j_2,j_3)}, \\
 j_3^\prime 	&= \frac{1}{2187} 	\, \frac{m(j_1,j_2,j_3)^2 \, \big( l + 4 \, k \,m\big)(j_1,j_2,j_3)}{q(j_1,j_2,j_3)} ,
\end{split}
\end{equation} 
where the polynomials $m, k, l, q$ with integer coefficients are given by
\begin{equation}
\begin{split}
  M(\tau) \, I_2^{-6}  &= 2^{-6} \, j_1^{-3} \, m(j_1,j_2,j_3) ,\\
  K(\tau) \, I_2^{-12} &= 2^{-12} \, j_1^{-6} \, k(j_1,j_2,j_3) ,\\
 L(\tau) \, I_2^{-18} &= 2^{-18} \, j_1^{-9} \,  l(j_1,j_2,j_3) ,\\
  Q(\tau) \, I_2^{-30}&= 2^{-63} \, j_1^{-15} \, q(j_1,j_2,j_3).
\end{split}
\end{equation} 
\endproof

The roots of the Satake sextic $y^2=f(x)$ determine by means of Equation~\eqref{Picard2} and Equation~\eqref{Theta2Satake} 
the Rosenhaim roots of an equivalent genus-two curve in Rosenhain normal form~(\ref{Rosen2}), with the original Igusa invariants, that is an element in $\M_2(2)$.
The permutation group $S_6$ acts on the roots $(x_1, \dots, x_6)$ by permutation. It is clear that the corresponding action on the sextic curve
in Rosenhain normal form~(\ref{Rosen2}) is the transition between the $720 = |S_6|$ choices of ratios of even theta constants which Picard's lemma allows for 
the three $\lambda$-parameters.

\section{Jacobian elliptic Kummer and Shioda-Inose surfaces}
In this section we describe the construction of the 
Kummer surface $\mathrm{Kum}(\operatorname{Jac}\mathcal{C})$ and the Shioda-Inose surface $\mathrm{SI}(\operatorname{Jac}\mathcal{C})$ 
together with the Jacobian elliptic fibrations on these surfaces which are relevant to the F-theory/heterotic string duality.  

\subsection{Jacobian elliptic fibrations}
A surface is called Jacobian elliptic fibration if it is a (relatively) minimal elliptic surface $\pi: \mathcal{X} \to \mathbb{P}^1$ over $\mathbb{P}^1$
with a distinguished section $S_0$. The complete list of possible singular fibers has been given by Kodaira~\cite{MR0184257}. 
It encompasses two infinite families $(I_n, I_n^*, n \ge0)$ and six exceptional cases $(II, III, IV, II^*, III^*, IV^*)$.
To each Jacobian elliptic fibration $\pi: \mathcal{X} \to \mathbb{P}^1$ there is an associated Weierstrass model $\bar{\pi}: \bar{\mathcal{X}\,}\to \mathbb{P}^1$ 
with a corresponding distinguished section $\bar{S}_0$ obtained by contracting
all components of fibers not meeting $S_0$. The fibers of $\bar{\mathcal{X}\,}$ are all irreducible whose singularities are all rational double points, and $\mathcal{X}$ is the minimal desingularization.
If we choose $t \in \mathbb{C}$ as a local affine coordinate on $\mathbb{P}^1$, we can present $\bar{\mathcal{X}\,}$ in the Weierstrass normal form
\begin{equation}
\label{Eq:Weierstrass}
 Y^2 = 4 \, X^3 - g_2(t) \, X - g_3(t) \;,
\end{equation}
where $g_2$ and $g_3$ are polynomials in $t$ of degree four and six, respectively, because $\mathcal{X}$ is a K3 surface. Type of singular fibers 
can then be read off from the orders of vanishing of the functions $g_2$, $g_3$ and the discriminant $\Delta= g_2^3 - 27 \, g_3^2$ 
at the various singular base values. Note that the vanishing degrees of $g_2$ and $g_3$ are always less or equal three and five, respectively,
as otherwise the singularity of $\bar{X}$ is not a rational double point.

For a family of Jacobian elliptic surfaces $\pi: \mathcal{X} \to \mathbb{P}^1$, the two classes in N\'eron-Severi lattice $\mathrm{NS}(\mathcal{X})$ associated 
with the elliptic fiber and section span a sub-lattice $\mathcal{H}$ isometric to the standard hyperbolic lattice $H$ with the quadratic form $q=x_1x_2$, and we have the following decomposition 
as a direct orthogonal sum
\begin{equation*}
 \mathrm{NS}(\mathcal{X}) = \mathcal{H} \oplus \mathcal{W} \;.
\end{equation*}
The orthogonal complement $T(\mathcal{X}) = \mathrm{NS}(\mathcal{X})^{\perp} \in H^2(\mathcal{X},\mathbb{Z})\cap H^{1,1}(\mathcal{X})$ is
called the transcendental lattice and carries the induced Hodge structure.  

\subsection{The Kummer surface}
For the Jacobian variety $\operatorname{Jac}(\mathcal{C})$ of a genus-two curve $\mathcal{C}$, let $\imath$ be the involution automorphism on $\operatorname{Jac}(\mathcal{C})$ 
given by $\imath: a \mapsto -a$.  The quotient, $\operatorname{Jac}(\mathcal{C})/\lbrace \mathbb{I}, \imath \rbrace$, is 
a singular surface with sixteen ordinary double points.  Its minimum resolution, $\operatorname{Kum}(\operatorname{Jac} \mathcal{C})$, is a special K3 surface called the \emph{Kummer surface} associated to
 $\operatorname{Jac} \mathcal{C}$. 
 
On the Kummer surface $\operatorname{Kum}(\operatorname{Jac} \mathcal{C})$, there are two sets of sixteen $(-2)$-curves, called nodes and tropes, which are either
the exceptional  divisors corresponding to the blowup of the 16 two-torsion points of the Jacobian $\operatorname{Jac}(\mathcal{C})$  or they  arise from the embedding of $\mathcal{C}$ into $\operatorname{Jac}(\mathcal{C})$  as symmetric theta divisors. These two sets of smooth rational curves have a rich symmetry, the so-called $16_6$-configuration, where each node intersects 
exactly six tropes and vice versa \cite{MR1097176}. 

Using curves and symmetries in the $16_6$-configuration one can define various elliptic fibrationson $\operatorname{Kum}(\operatorname{Jac} \mathcal{C})$, 
since all irreducible components of a reducible fiber in an elliptic fibration 
are $(-2)$-curves \cite{MR0184257}. In fact, for the Kummer surface of a generic curve of genus two all inequivalent elliptic fibrations were determined explicitly by Kumar in \cite{MR3263663}. In particular, Kumar computed elliptic parameters and  Weierstrass equations for all twenty five different fibrations that appear, and analyzed the reducible fibers and  Mordell-Weil lattices.

\subsubsection{A first elliptic fibration on $\mathrm{Kum}(\operatorname{Jac}\mathcal{C})$}
Given a genus-two curve $\mathcal{C}$ defined by a sextic $Y^2 =F(X)$,
the Jacobian variety $\operatorname{Jac}(\mathcal{C})$ is birational to the symmetric product of two copies of $\mathcal{C}$, i.e., 
$(\mathcal{C}\times\mathcal{C})/\{ \mathbb{I}, \pi \}$, where we have set
$\pi(X_1)=X_2$ and $\pi(Y_1)=Y_2$. The function field is the sub-field of $\mathbb{C}[X_1,X_2,Y_1,Y_2]$ such that $Y_i^2 =F(X_i)$ for $i=1,2$ which is fixed under 
$\pi$.

The Kummer surface $\mathrm{Kum}(\operatorname{Jac}\mathcal{C})$ is birational to the quotient $\operatorname{Jac}(\mathcal{C})/\{ \mathbb{I}, \imath \}$ 
with $\imath(X_i)=X_i$ and $\imath(Y_i)=-Y_i$ for $i=1,2$. Its function
field is the sub-field of $\mathbb{C}[X_1,X_2,Y_1,Y_2]$ with $Y_i^2 =F(X_i)$ for $i=1,2$ which is fixed under both $\pi$ and $\imath$. 
Thus, the function field of $\mathrm{Kum}(\operatorname{Jac}\mathcal{C})$ is generated by $Y=Y_1Y_2$, $t=X_1X_2$, and $X=X_1+X_2$.
We have the following lemma:
\begin{lem}
The function field of the Kummer surface $\mathrm{Kum}(\operatorname{Jac}\mathcal{C})$ for the genus-two curve $\mathcal{C}$ given by the sextic~(\ref{Rosen2}) 
is generated by $Y, X, t$  subject to the relation $\mathcal{K}(Y,X,t)=0$ with
\begin{equation}
\label{kummer2}
\mathcal{K}(Y,X,t) = Y^2 - t \,  \big( 1 -  X + t\big) \, \big( \lambda_1^2 - \lambda_1 \, X + t\big) \, \big( \lambda_2^2 - \lambda_2 \, X + t\big) \, \big( \lambda_3^2 - \lambda_3 \, X + t\big) .
\end{equation}
\end{lem}

Equation~(\ref{kummer2}) defines a Jacobian elliptic fibration
$\bar{\pi}: \bar{\mathcal{X}\,} \to \mathbb{P}^1$ with a distinguished section $\bar{S}_0$  on $\mathcal{X}=\mathrm{Kum}(\operatorname{Jac}\mathcal{C})$ by choosing $t$ as 
the elliptic parameter and the point at infinity in each fiber for the section. In fact, this fibration is well-known and labeled fibration `1' in  \cite{MR3263663}. 
The following lemma is immediate and follows by comparison with the explicit results in \cite{MR3263663}. 
\begin{lem}
\label{lem:EllLeft}
Equation~(\ref{kummer2}) is a Jacobian elliptic fibration $\bar{\pi}: \bar{\mathcal{X}} \to \mathbb{P}^1$ with 6 singular fibers of type $I_2$, 
two singular fibers of type $I_0^*$, and the Mordell-Weil group of sections $\operatorname{MW}(\bar{\pi})=(\mathbb{Z}/2)^2\oplus \langle 1 \rangle$.
\end{lem}

The sextic curve is recovered directly from the Jacobian elliptic fibration by the following corollary:
\begin{cor}
A sextic associated with the genus-two curve $\mathcal{C}$ is recovered from the Jacobian elliptic fibration~(\ref{kummer2}) on $\mathrm{Kum}(\operatorname{Jac}\mathcal{C})$ 
by letting $X \to \infty$ while keeping 
$t/X=\xi$, $Y^2/X^5=\eta^2$ fixed, i.e.,
\begin{equation}
\lim_{\epsilon \to 0} \epsilon^{10} \, \mathcal{K}\left(Y=\frac{\eta}{\epsilon^5}, X= \frac{1}{\epsilon^2}, t= \frac{\xi}{\epsilon^2} \right) = \eta^2 - F(\xi) \;.
\end{equation}
\end{cor}
\begin{proof}
The proof follows from an explicit computation.
\end{proof}

\subsubsection{A second elliptic fibration on $\mathrm{Kum}(\operatorname{Jac}\mathcal{C})$}
There is another elliptic fibration $\mathcal{X} \to \mathbb{P}^1$ on $\mathrm{Kum}(\operatorname{Jac}\mathcal{C})$ which is more relevant for us,
labeled fibration `23' in  \cite{MR3263663}. Kumar proved the following \cite{MR2427457}:
\begin{prop}
A Jacobian elliptic fibration $\bar{\pi}: \bar{\mathcal{X}} \to \mathbb{P}^1$ with 6 singular fibers of type $I_2$, one fiber of type $I_5^*$,
one fiber of type $I_1$, and a Mordell-Weil group of sections $\operatorname{MW}(\bar{\pi})=\mathbb{Z}/2$ is given by the Weierstrass equation
\begin{equation}
\label{KumFib2}
\begin{split}
Y^2 = & \, X^3  - 2 \, \left( t^3 - \frac{I_4}{12} t+ \frac{I_2 I_4-3 I_6}{108} \right) \,X^2 \\
& + \left( \left(t^3 - \frac{I_4}{12} t + \frac{I_2 I_4 - 3 I_6}{108}\right)^2 + I_{10} \left(t- \frac{I_2}{24}\right)\right) \, X \;.
\end{split}
\end{equation}
\end{prop}
An immediate corollary is the following:
\begin{cor}
The positions of the $I_2$ fibers in the elliptic fibration~(\ref{KumFib2}) on the Kummer surface $\mathrm{Kum}(\operatorname{Jac}\mathcal{C})$ are given by the roots of the polynomial
\begin{equation}
\label{SatakeSextic_a}
 \left(t^3 - \frac{I_4}{12} t + \frac{I_2 I_4 - 3 I_6}{108}\right)^2 + I_{10} \left(t- \frac{I_2}{24}\right) =0 \;,
\end{equation}
or equivalently by
\begin{equation}
\label{SatakeSextic_b}
 \left(t^3 - \frac{\psi_4}3 t + \frac{2 \, \psi_6}{27}\right)^2 -2^{14} \,  \left(\chi_{10} \, t + \chi_{12}\right) =0 \;.
\end{equation}
Equivalently, the loci of $I_2$ fibers form the ramification locus of the Satake sextic~(\ref{SatakeSextic})
if we set $t=-x/3$.
\end{cor}

\subsection{The Shioda-Inose surface}
A K3 surface $\mathcal{Y}$ has a Shioda-Inose structure if it admits an involution fixing the holomorphic two-form, such that the 
quotient is the Kummer surface $\operatorname{Kum}(\mathbf{A})$ of a principally polarized abelian surface $\mathbf{A}$ and the
rational quotient map $p: \mathcal{Y} \dashrightarrow \operatorname{Kum}(\mathbf{A})$ of degree two
induces a Hodge isometry\footnote{A Hodge isometry between two transcendental lattices is an isometry preserving the Hodge structure. } 
between the transcendental lattices $T(\mathcal{Y})(2)$\footnote{The notation $T(\mathcal{Y})(2)$ indicates 
that the bilinear pairing on the transcendental lattice $T(\mathcal{Y})$ is multiplied by $2$.}
 and  $T(\operatorname{Kum} \mathbf{A})$. 
Morrison proved that a K3 surface $\mathcal{Y}$ admits a Shioda-Inose structure if and only if there exists a Hodge isometry
between the following transcendental lattices $T(\mathcal{Y}) \cong T(\mathbf{A})$. 

The Shioda-Inose $\mathcal{Y}$ of the Kummer surface $\operatorname{Kum}(\operatorname{Jac} \mathcal{C})$ for a generic genus-two  curve $\mathcal{C}$ is a K3 surface of Picard-rank 17 
and has a  transcendental lattice  isomorphic to $H \oplus H \oplus \langle -2 \rangle$
 by Morrison's criterion. It was shown in \cite{MR2306633} that for fixed $\mathcal{C}$ this K3 surface $\mathcal{Y}$ is in fact unique.
In the following, we always let  $\mathcal{Y}=\mathrm{SI}(\operatorname{Jac}\mathcal{C})$ be this K3 surface with Shioda-Inose structure.

Clingher and Doran proved  in~\cite{MR2824841} that as the genus-two curve $\mathcal{C}$ varies the K3 surface $\mathcal{Y}$ admits a birational model isomorphic to a quartic surface 
with canonical $H \oplus E_8(-1)\oplus E_7(-1)$\footnote{Here, $E_8(-1)$ and $E_7(-1)$ 
are the negative definite lattice associated with the exceptional root systems of $E_8$ and $E_7$, respectively.}  lattice polarization\footnote{A lattice polarization 
of a K3 surface $\mathcal{Y}$ is a primitive embedding of a lattice $L' \hookrightarrow L = H_2(\mathcal{Y}, \mathbb{Z})$ such that the image of 
$L'$ lies in the N\'eron-Severi group $\operatorname{NS}(\mathcal{Y}) = L \cap H^{1,1}(\mathcal{Y})$ and contains a pseudo-ample class.} 
that fits into  the following four-parameter family in $\mathbb{P}^3$ \cite{MR2824841}*{Eq.~\!(3)} given in terms of the
variables $[\mathbf{W}:\mathbf{X}:\mathbf{Y}:\mathbf{Z}]\in \mathbb{P}^3$ by the equation
\begin{equation}
\label{Inose}
\mathbf{Y}^2\mathbf{ZW}-4\, \mathbf{X}^3\mathbf{Z}+3\, \alpha \, \mathbf{XZW}^2 + \beta \, \mathbf{ZW}^3 + \gamma \, \mathbf{XZ}^2 \mathbf{W} -\frac{1}2(\delta \, 
\mathbf{Z}^2\mathbf{W}^2+\mathbf{W}^4)=0.
\end{equation}
They also found the parameters $(\alpha,\beta,\gamma,\delta)$ in terms of the standard even Siegel modular forms $\psi_4, \psi_6, \chi_{10}, \chi_{12}$ (cf.~\cite{MR0141643}) given by
\begin{equation}
 (\alpha,\beta,\gamma,\delta) = \left(\psi_4, \psi_6, 2^{12}3^5 \, \chi_{10}, 2^{12}3^6 \, \chi_{12}\right) \;,
\end{equation}
or, equivalently, in terms of the Igusa-Clebsch invariants using Equations~(\ref{invariants}) by
\begin{equation}
 (\alpha,\beta,\gamma,\delta) = \left(\frac{1}4 I_4, \frac{1}{8} I_2 \, I_4 - \frac3{8} I_6, - \frac{243}4 I_{10}, \frac{243}{32} I_2\, I_{10}\right) \;,
\end{equation}
where $I_n$ for $n=2,4,6,10$ are the Igusa invariants of the sexic curve~(\ref{Rosen2}) defining the genus-two curve $\mathcal{C}$ if $I_{10} \not = 0$.
The Shioda-Inose surface $\mathcal{Y}=\mathrm{SI}(\operatorname{Jac}\mathcal{C})$ of a generic genus-two curve $\mathcal{C}$ admits two Jacobian elliptic fibrations
realizing the two inequivalent ways of embedding of $H$ into the lattice $H \oplus E_8(-1)\oplus E_7(-1)$.
These two elliptic fibrations were described in \cites{MR2427457, MR2935386}.
A similar picture was developed in earlier work for the case of a $H \oplus E_8(-1) \oplus E_8(-1)$ lattice polarization in \cites{MR2369941,MR2279280} that generalized a special two-parameter 
family of K3 surfaces introduced by Inose in \cite{MR578868}. 

From the point of view of K3 geometry, if the periods are preserved by
a reflection of $\delta$ with $\delta^2=-2$, then $\delta$ must belong
to the N\'eron-Severi lattice of the K3 surface.  That is, the
lattice $H \oplus E_8(-1)\oplus E_7(-1)$ must be enlarged
by adjoining $\delta$.  It is not hard to show (using methods of
\cite{MR525944}, for example), that there are only two ways this enlargement
can happen (if we have adjoined a single element only):  either the lattice is extended to 
$H \oplus E_8(-1)\oplus E_8(-1)$ or it is extended to $H\oplus E_8(-1)\oplus E_7(-1) \oplus
\langle -2 \rangle$.  

\subsubsection{The alternate fibration}
The first Jacobian elliptic fibration on (\ref{Inose}), called the \emph{alternate fibration}, has two disjoint sections and a singular fiber of Kodaira-type $I_{10}^*$.
For convenience, let us introduce the parameters $(a,b,c,d,e)$ given by
\begin{equation}
\label{moduli_abcde}
a =- \dfrac{I_4}{12}, \; b=\dfrac{I_2 \, I_4 - 3 \, I_6}{108},\; c=-1, \; d=\dfrac{I_2}{24}, \; e=\dfrac{I_{10}}4.
\end{equation}
The alternate fibration is obtained by setting
\begin{equation}
\label{transfo_alt}
 \mathbf{X} = \dfrac{t \, x^3}{2^{29} \, 3^5} \;, \quad \mathbf{Y}=\dfrac{\sqrt6\, i \, x^2  y}{2^{29} \, 3^5} \;,\quad \mathbf{W}=-\dfrac{x^3}{2^{28} \, 3^6} \;, \quad \mathbf{Z}= \dfrac{x^2}{2^{28} \, 3^9} \;,
\end{equation}
in Equation~(\ref{Inose}), and given in Weierstrass form by
\begin{equation}
\label{WEq.bak.alt}
 y^2 = x^3 + (t^3 + a t + b) \, x^2 + e \, (ct+d) \, x \;.
\end{equation}
The discriminant of Equation~(\ref{WEq.bak.alt}) is given by
\begin{equation}
 \Delta = 16\, e^2 \,  \left( ct+d \right)^2 \, \Big( (t^3 + a t + b)^2 - 4 \, e \, (ct+d) \Big) \;.
\end{equation}
The fibration (\ref{WEq.bak.alt}) has special fibers of Kodaira-types $I_{10}^*$ and $I_2$, and  six fibers of Kodaira-type $I_1$, and a second two-torsion section $\bar{S}_1: (x,y)=(0,0)$.
This proves the following:
\begin{prop}
\label{Prop:alternate}
For a generic genus-two curve $\mathcal{C}$ there is a Jacobian elliptic fibration $\bar{\pi}_{\text{alt}}: \bar{\mathcal{Y}\,} \to \mathbb{P}^1$ on $\mathcal{Y}=\mathrm{SI}(\operatorname{Jac}\mathcal{C})$
given by Equation~(\ref{WEq.bak.alt}) with 6 singular fibers of type $I_1$, one fiber of type $I_{10}^*$, and one fiber of type $I_2$, and a Mordell-Weil group of sections 
$\operatorname{MW}(\bar{\pi})=\mathbb{Z}/2$  with elliptic parameter $t=t_{\text{alt}}$.
\end{prop}
The following corollary is crucial:
\begin{cor}
\label{cor:position_nodes}
The positions of the $I_1$ fibers in the Jacobian elliptic fibration~(\ref{KumFib2})
are given by the roots of the polynomial
\begin{equation}
\label{Satake_b}
    f(t) = (t^3 + a t + b)^2 - 4 \, e \, (ct+d) = 0 \;.
\end{equation}
Equivalently, the loci of $I_1$ fibers form the ramification locus of the Satake sextic~(\ref{SatakeSextic})
if we set $t=-x/3$.
\end{cor}
Given the discussion at  the end of previous section, the following corollaries are immediate:
\begin{cor}
\label{cor:Qvanish}
For the Jacobian elliptic fibration~(\ref{WEq.bak.alt}) two $I_1$ fibers coalesce and form a fiber of type $I_2$
if and only if the discriminant of the Satake sextic~(\ref{Satake_b}) vanishes, i.e., $\Delta_{f} = Q  = 0$,
or, equivalently,
\begin{equation}
\begin{split}
0  = {e}^3 \Big( 16\,{a}^{7}{c}^2d-16\,{a}^6b{c}^3+16\,{a}^{
5}{c}^4e+16\,{a}^6{d}^3\\+216\,{a}^4{b}^2{c}^2d+888\,{a}^4
{c}^2{d}^2e-216\,{a}^3{b}^3{c}^3-3420\,{a}^3b{c}^3de+
2700\,{a}^2{b}^2{c}^4e\\+4125\,{a}^2{c}^4d{e}^2-5625\,ab{c}^
5{e}^2+3125\,{c}^6{e}^3+216\,{a}^3{b}^2{d}^3+864\,{a}^{3
}{d}^4e\\-2592\,{a}^2bc{d}^3e+729\,a{b}^4{c}^2d-5670\,a{b}^2
{c}^2{d}^2e+16200\,a{c}^2{d}^3{e}^2-729\,{b}^5{c}^3\\+6075
\,{b}^3{c}^3de-13500\,b{c}^3{d}^2{e}^2+729\,{b}^4{d}^3-
5832\,{b}^2{d}^4e+11664\,{d}^5{e}^2 \Big).
\end{split}
\end{equation}
Equivalently, $\Delta_{f} =0$ iff the lattice polarization $H \oplus E_8(-1)\oplus E_7(-1)$ of the family~(\ref{WEq.bak.alt}) extends to 
$H\oplus E_8(-1)\oplus E_7(-1) \oplus \langle -2 \rangle$.  
\end{cor}
\begin{rem}
We remark that an $I_1$ fiber will coalesce with the $I_2$ fiber to form a fiber of type $III$ if and only if
$$
 e^3 \, (a \, c^2 \, d-b \, c^3+d^3) =  
 2 \, \psi_6 \, \chi_{10}^3 +9 \, \psi_4 \, \chi_{10}^2 \, \chi_{12} -27 \, \chi_{12}^3 
 =0.
$$ 
However, it is easy to show that this does not change the lattice polarization of the family~(\ref{WEq.bak.alt}).
\end{rem}
If we use a normalization consistent with F-theory \cite{MR3366121} and set
\begin{equation}
 \mathbf{X} = \dfrac{t \, x^3}{2^9 \, 3^5} \;, \quad \mathbf{Y}=\dfrac{x^2 \, y}{2^{15/2} \, 3^{9/2}} \;,\quad \mathbf{W}=\dfrac{x^3}{2^{10} \, 3^6} \;, \quad \mathbf{Z}= \dfrac{x^2}{2^{16} \, 3^9} \;,
\end{equation}
and obtain from Equation~(\ref{Inose}) the Jacobian elliptic fibration
\begin{equation}
\label{WEq.bak.alt2}
 y^2 = x^3  + \left( t^3 - \frac{\psi_4}{48} \, t - \frac{\psi_6}{864} \right) \, x^2 -  \Big( 4 \, \chi_{10}  \, t - \chi_{12} \Big) \, x\;,
\end{equation}
which is equivalent to Equation~(\ref{WEq.bak.alt2}) for $\chi_{10}(\tau)\not =0$. Equation~(\ref{WEq.bak.alt2}) remains well-defined for
$\chi_{10}(\tau)=0$ when the principally polarized abelian surface $\mathbf{A}_{\tau} $ degenerates to a product of two elliptic 
curves  $\mathbf{A}_{\tau} =\mathcal{E}_{\tau_1} \times \mathcal{E}_{\tau_2}$. It follows:
\begin{cor}
\label{cor:Chi10vanish}
For the Jacobian elliptic fibration~(\ref{WEq.bak.alt2}) the two fibers of type $I_2$ and $I_{10}^*$ coalesce and form a fiber of type $I_{12}^*$
if and only if the discriminant of the sextic~(\ref{Rosen2}) vanishes, i.e., 
$$
 \Delta_{F}  = \chi_{10} =  e = 0 \;.
$$
Equivalently, $\Delta_{F} =0$ iff the lattice polarization $H \oplus E_8(-1)\oplus E_7(-1)$ of the family~(\ref{WEq.bak.alt2}) extends to 
$H\oplus E_8(-1)\oplus E_8(-1)$.  
\end{cor}

Using fibration~(\ref{WEq.bak.alt}) and fibration~(\ref{KumFib2}), the rational quotient map 
$$p: \mathcal{Y}=\mathrm{SI}(\operatorname{Jac}\mathcal{C}) \dashrightarrow \mathcal{X}=\operatorname{Kum}(\operatorname{Jac}\mathcal{C})$$ 
is realized as a fiberwise two-isogeny between elliptic surfaces, also known as a \emph{Van Geemen-Sarti involution}. Together with the dual isogeny one obtains a chain of rational maps  $\mathcal{Y} \dashrightarrow \mathcal{X} \dashrightarrow \mathcal{Y}$ called a \emph{Kummer sandwich} in \cite{MR2279280}.
In fact, the translation of the elliptic fiber $\mathcal{E}=\mathcal{Y}_t$ in Equation~(\ref{WEq.bak.alt}) by a two-torsion point
$\bar{S}_1: (x,y)=(0,0)$ yields the two-isogeneous fiber $\mathcal{E}'= \mathcal{E}/\lbrace S_0, S_1 \rbrace$ given by
\begin{equation}
\label{KumFib2b}
 Y^2 = X^3 - 2 (t^3 + a t + b) \, X^2 + \Big( (t^3 + a t + b)^2 - 4 \, e \, (ct+d)\Big) \, X,
\end{equation}
which is precisely the fibration~(\ref{KumFib2}). The fibrewise isogeny $\mathcal{E} \to \mathcal{E}'=\mathcal{X}_t$ is given by
\begin{equation}
\label{Isogeny}
 (x,y) \mapsto (X,Y)=\left( \frac{y^2}{x^2}, \frac{y \, (e \, (ct+d) -x^2)}{x^2} \right),
\end{equation}
and the dual isogeny $\mathcal{E}' \to \mathcal{E}$ is given by
\begin{equation}
\label{Isogeny_dual}
 (X,Y) \mapsto (x,y)=\left( \frac{Y^2}{4 X^2}, \frac{Y \, \Big(  (t^3 + a t + b)^2 - 4 e  (ct+d) -X^2\Big)}{X^2} \right) \;.
\end{equation}
The resulting \emph{Nikulin involution} $\varphi$ on the K3 surface $\mathcal{Y}$, i.e., the automorphism of order two preserving
the holomorphic two-form, in this case the two-form $dt \wedge dx/y$, is given by adding to a point the two-torsion section in each generic fiber $\mathcal{Y}_t$, i.e., 
\begin{equation*}
\label{eq:Nik_alt}
 (x,y) \mapsto (x,y) \overset{.}{+} \bar{S}_1 = \left(  \frac{e (ct + d)}{x}, - \frac{y \,e (ct + d)}{x^2}\right) .
\end{equation*}
A fiber of type $I_1$ is a rational curve
with a node, whereas a fiber of type $I_2$ looks like two copies of $\mathbb{P}^1$ intersecting in two distinct points.
The involution $\varphi$ is free on the generic fibers and
has exactly $8$ fixed points, namely the nodal points on the
$I_1$ fibers and the intersecting points on the $I_2$ fiber.  We have the following corollary:
\begin{cor}
\label{cor:position_nodes2}
The positions of the $I_1$ fibers in the Jacobian elliptic fibration~(\ref{KumFib2})
are contained in the fixed point set of the Nikulin involution $\varphi$.
\end{cor}

\subsubsection{The standard fibration}
The second elliptic fibration, called the \emph{standard fibration}, is the Jacobian elliptic fibration with two distinct special fibers of Kodaira-types $II^*$ and $III^*$, respectively. 
By setting
\begin{equation}
\label{transfo_std}
 \mathbf{X} = - \dfrac{2^7\, \chi_{10}^3 \, t \, x}{3^5} \;, \quad \mathbf{Y}=\dfrac{2^7 \, \sqrt6 \, i \, \chi_{10}^3 \, y}{3^5} \;,\quad \mathbf{W}=\dfrac{2^8 \, \chi_{10}^3 \, t^3}{3^6} \;, \quad \mathbf{Z}= \dfrac{\chi_{10}^2 \, t^2}{2^4 \, 3^9} \;,
\end{equation}
in Equation~(\ref{Inose}) we obtain
\begin{equation}
\label{WEq.bak}
 y^2 = x^3 + t^3 \, (a \, t + c) \, x + t^5 \, (e \, t^2 + b \, t +d) \;.
\end{equation}
The fibration (\ref{WEq.bak}) was investigated in \cite{MR2427457}*{Theorem~\!11}.
The birational transformation between the standard and the alternate fibration is given by
\begin{equation}
 \Big( t, x, y \Big)_{\text{std}} = \Big( \frac{x}{e}, \frac{t x^2}{e^2}, - \frac{x^2 y}{e^3} \Big)_{\text{alt}} \;,
\end{equation}
and combining it with Equation~(\ref{eq:Nik_alt}) recovers the Nikulin involution for the standard elliptic fibration
given by (cf.~\cite{MR2427457}*{Theorem~\!11})
\begin{equation}
 \Big( t, x, y \Big) \mapsto \left( \frac{cx + d t^2}{e t^3}, \frac{x (cx + d t^2)^2}{e^2 t^8}, - \frac{y (cx + d t^2)^3}{e^3 t^{12}} \right).
\end{equation}

It then follows:
\begin{prop}
\label{Prop:standard}
For a generic genus-two curve $\mathcal{C}$ there is a Jacobian elliptic fibration $\bar{\pi}_{\text{std}}: \bar{\mathcal{Y}\,} \to \mathbb{P}^1$ on 
$\mathcal{Y}=\mathrm{SI}(\operatorname{Jac}\mathcal{C})$ given by Equation~(\ref{WEq.bak}) with 5 singular fibers of type $I_1$, one fiber of 
type $II^*$, one fiber of type $III^*$, and a trivial Mordell-Weil group with elliptic parameter $t=t_{\text{std}}$.
\end{prop}

For the standard fibration, there are statements analogous to Corollary~\ref{cor:Qvanish} or Corollary~\ref{cor:Chi10vanish}
when two $I_1$ fibers are coalescing to form a fiber of type $II$ or a fiber from type $III^*$ goes to type $II^*$,
respectively \cite{MR3366121}.

\subsection{Relation to string theory}
In string theory a nontrivial connection appears as the eight-dimensional manifestation of a phenomenon called \emph{F-theory/heterotic string duality}. 
This correspondence leads to a geometric picture that links together moduli spaces for two seemingly distinct types of geometrical objects: Jacobian elliptic fibrations on K3 surfaces 
and flat bundles over elliptic curves \cite{MR1675162}.

In string theory compactifications of the so-called type-IIB string in which the complex coupling varies over a base are generically referred to as F-theory.
The simplest such construction corresponds to K3 surfaces that are elliptically fibered over $\mathbb{P}^1$ with a section, in physics 
equivalent to type-IIB string theory compactified on $\mathbb{P}^1$ and hence eight-dimensional in the presence of 7-branes \cite{MR1416354}.
In this way, a Jacobian elliptic K3 surface with elliptic fibers $\mathcal{E}_{\tau }=\mathbb{C}/(\mathbb{Z}\oplus \mathbb{Z}\tau )$ defines
an F-theory vacuum in eight dimensions where the complex-valued scalar field $\tau$ of the type-IIB string theory is now allowed to be multi-valued.
The Kodaira-table of singular fibers gives a precise dictionary between the characteristics of the Jacobian elliptic fibrations and the content of the 
7-branes present in the physical theory.

To make contact with the F-theory/heterotic string duality one considers Jacobian elliptic fibrations on a special K3 surface, namely the Shioda-Inose
surface $\mathrm{SI}(\operatorname{Jac}\mathcal{C})$ of the principally polarized abelian surface $\operatorname{Jac}(\mathcal{C})$
where $\mathcal{C}$ is a generic genus-two curve.  The K3 surface $\mathcal{Y}=\mathrm{SI}(\operatorname{Jac}\mathcal{C})$ carries a Nikulin involution $\varphi$ such that the quotient
$\mathcal{Y}/\lbrace\mathbb{I}, \varphi\rbrace$ is birational to the Kummer surface  $\mathrm{Kum}(\operatorname{Jac}\mathcal{C})$ and we have a Hodge-isometry 
between the transcendental lattices $T(\mathcal{Y}) \cong T(\operatorname{Jac}\mathcal{C})$.
 In this way, a one-to-one correspondence between two different types of surfaces with the same Hodge-theoretic data is established:
 the principally polarized abelian surfaces $\operatorname{Jac}\mathcal{C}$ and the algebraic K3 surfaces $\mathcal{Y}$
 polarized by the rank-17 lattice $H \oplus E_8(-1)\oplus E_7(-1)$.

To see the connection to the heterotic string theory, let us first consider the limit as the Jacobian variety degenerates to a product of two elliptic curves $%
\mathcal{E}_{1}\times \mathcal{E}_2$ and the involved K3 surfaces have Picard-rank 18, obtained by letting $\chi_{10}\to 0$.
 This limit describes a well-understood case of the
F-theory/heterotic string duality in the absence of any additional bundle data given by so-called Wilson lines. In fact,
the moduli space of Jacobian elliptic K3 surfaces with $H\oplus  E_8(-1)\oplus E_8(-1)$ lattice polarization is identified with the
moduli space of the heterotic string vacua with gauge algebra $\mathfrak{e}_8 \oplus \mathfrak{e}_8$ and 
$\mathfrak{so}(32)$, respectively, compactified on a two-torus $T^2$ (cf.~\!\cites{MR1416960}). If any flat connection on $T^2$ is assumed to be trivial, 
the only two moduli of such a string theory, i.e,  the K\"{a}hler metric and the $B$-field of $T^2$, identify the torus with the elliptic
curves $\mathcal{E}_{1}$ and $\mathcal{E}_2$, respectively. Notice that the existence of two inequivalent elliptic fibrations, the standard and the alternate
fibration, is essential and corresponds to the two possible gauge groups of the heterotic string.

The first author together with David Morrison studied in \cite{MR3366121} the \emph{non-geometric} heterotic string
compactified on $T^2$ that produces an eight-dimensional effective theory corresponding to the Jacobian elliptic K3 surfaces with Picard-rank 17 when $\chi_{10}\not=0$. 
The corresponding heterotic models were called non-geometric because the K\"ahler and 
complex structures on $T^2$, and the Wilson line values, are not distinguished
but instead are mingled together. The fibration in Equation~(\ref{WEq.bak})
then describes a model dual to the $\mathfrak{e}_8\oplus \mathfrak{e}_8$ heterotic string, with an unbroken gauge algebra of $\mathfrak{e}_8\oplus \mathfrak{e}_7$ 
ensuring that only a single Wilson line expectation value is nonzero and all remaining Wilson lines values associated to the $E_8(-1)\oplus E_7(-1)$
sublattice be trivial. Similarly, the fibration in Equation~(\ref{WEq.bak.alt}) 
gives the analogous story for the $\mathfrak{so}(32)$ heterotic string:  the fibration
describes a model dual to the $\mathfrak{so}(32)$ heterotic string, with an unbroken gauge algebra 
of $\mathfrak{so}(28)\oplus \mathfrak{su}(2)$.
By a result of Vinberg \cite{MR3235787} and its interpretation in string theory given in \cite{MR3366121}, the function field of
the Narain moduli space of these heterotic theories  turns out to be generated by the ring of Siegel modular forms of 
{\it even weight.} This is the physical manifestation of why the fibrations~(\ref{WEq.bak}) and (\ref{WEq.bak.alt}) only depend
on the polynomial ring in the four free generators of degrees $4$, $6$, $10$ and $12$ given by the even Siegel modular forms.\footnote{In contrast, Igusa showed in \cite{MR0229643} that for the full
ring of modular forms, one needs an additional generator $\chi_{35}$ which is algebraically dependent on the others.}

Generic non-geometric compactification constructed from the family of 
lattice-polarized K3 surfaces in Equation~(\ref{Inose}) will have two types of five-branes analogous to a single D7-brane in F-theory.  
From the heterotic side, these five-brane solitons are easy to see:  in the situation of Corollary~\ref{cor:Qvanish} we have an additional
gauge symmetry enhancement by a factor of $\mathfrak{su}(2)$, and the parameters
of the theory will include a Coulomb branch with Weyl group $W_{\mathfrak{su}(2)}=\mathbb{Z}_2$. Therefore, there is a five-brane
solution with a $\mathbb{Z}_2$ ambiguity when encircling the location
in the moduli space of enhanced gauge symmetry.
The other five-brane solution is similar:
in the situation of Corollary~\ref{cor:Chi10vanish} the gauge group enhances to $\mathfrak{so}(32)$ gauge symmetry
with a similar $\mathbb{Z}_2$ ambiguity.

Further specializations of the multi-parameter family of K3 surfaces in Equation~(\ref{Inose})
are obtained from degenerations of the underlying genus-two curves. As we have seen, the parameters 
in Equation~(\ref{Inose}) are Siegel modular forms of even degree or,
equivalently, the Igusa-Clebsch invariants of a binary sextic. Namikawa and 
Ueno gave a geometrical classification of all (degenerate) fibers in pencils of curves of genus two in \cite{MR0369362}.
For each such pencil allowed by their classification one can now apply the heterotic/F-theory duality map to 
express the heterotic background in terms of F-theory. Each resulting F-theory compactification will be a family of 
Jacobian elliptic K3 surfaces. Notice that any such degenerating pencil of genus-two curves is not the description of a heterotic model itself, 
but rather a computational tool for providing an interesting class of degenerations and their associated five-branes.
Moreover, the F-theory background dual to a given five-brane defect on the heterotic side will in general be highly singular. For some of theses cases the singularities can be resolved
by performing a finite number of blow-ups in the base, and the resulting smooth geometry was constructed in \cite{Font:2016aa}.

Conversely, the combination of Proposition~\ref{Satake6} and Corollary~\ref{cor:position_nodes} give a computational recipe for how a degenerating pencil of 
genus-two curves is obtained  from an F-theory background dual to the $\mathfrak{so}(32)$ string with only one non-vanishing Wilson line.
In comparison, the work of the authors in \cite{Malmendier:2016aa} always allows for the construction of an explicit 
pencil of sextic curves given any family of Igusa invariants over a quadratic extension of the full ring of modular forms.
However, this construction does not use the F-theoretic data of the $\mathfrak{so}(32)$ string background, i.e., the Jacobian elliptic fibration,
and requires lifting of the family to a covering space of the moduli space. 
In contrast, Corollary~\ref{cor:position_nodes} shows that the Satake sextic is inherently 
manifest in the Jacobian elliptic fibration~(\ref{WEq.bak.alt}).

We rephrase Corollary~\ref{cor:position_nodes} according to the discussion in this section as follows:
\begin{cor}
\label{cor:position_nodes_b}
The positions of the 7-branes with string charge $(1,0)$ in the F-theory model, dual to the $\mathfrak{so}(32)$ heterotic string with an unbroken gauge algebra 
of $\mathfrak{so}(28)\oplus \mathfrak{su}(2)$ and only a single non-vanishing Wilson line expectation value and no additional gauge-extension, are given by the loci of $I_1$ fibers in 
the Jacobian elliptic fibration~(\ref{WEq.bak.alt}) on the Shioda-Inose surface $\mathrm{SI}(\operatorname{Jac}\mathcal{C})$ of a generic genus-two curve $\mathcal{C}$
and form the ramification locus of the Satake sextic~(\ref{SatakeSextic_a}) corresponding to $\mathcal{C}$, or, equivalently,
the genus-two component of the fixed point set of the Nikulin involution $\varphi$ on $\mathrm{SI}(\operatorname{Jac}\mathcal{C})$.
\end{cor}

\begin{rem}
The section $(x,y)=(0,0)$ defines an element of order  $2$ in the Mordell-Weil group of the Jacobian elliptic fibration~(\ref{WEq.bak.alt}).  It follows
as in \cites{MR1416960,MR1643100} that the actual gauge group of this heterortic model
is $(\operatorname{Spin}(28)\times SU(2))/\mathbb{Z}_2$.
\end{rem}

In turn, the roots of the Satake sextic then determine a sextic curve~(\ref{Rosen2}) with full level-two structure 
by using Equation~(\ref{Theta2Satake}) and Equation~(\ref{Picard2}). 

\begin{landscape}
\section{Appendix}
The Igusa-Clebsch invariants for the curve~(\ref{Rosen2}) in Rosenhain normal form
are given by the following expressions:
\begingroup\makeatletter\def\f@size{7}\check@mathfonts\def\maketag@@@#1{\hbox{\m@th\large\normalfont#1}}%
 \begin{align*}
I_2 & = 40\,\lambda_1\lambda_2\lambda_3-16\, \left( 1+ \lambda_1+ \lambda_2 + \lambda_3 \right)  \left( \lambda_1\lambda_2\lambda_3+\lambda_2\lambda_1+\lambda_3
\lambda_1+\lambda_2\lambda_3 \right) +6\, \left( \lambda_2\lambda_1+\lambda_3\lambda_1+\lambda_2\lambda_3+\lambda_1+\lambda_2+\lambda_3 \right) ^2,\\
I_4 &= 
-12\, \left( \lambda_1+\lambda_2+\lambda_3 \right) ^3
\lambda_1\lambda_2\lambda_3+4\, \left( \lambda_1+
\lambda_2+\lambda_3 \right) ^2 \left( \lambda_2\lambda_1+\lambda_3\lambda_1+\lambda_2\lambda_3 \right) ^2-4\, \left( \lambda_1+\lambda_2+\lambda_3 \right) ^2
 \left( \lambda_2\lambda_1+\lambda_3\lambda_1+\lambda_2\lambda_3 \right) \lambda_1\lambda_2\lambda_3\\
& +4\, \left( \lambda_1+\lambda_2+\lambda_3 \right) ^2\lambda_1^2\lambda_2^2\lambda_3^2+12\, \left( \lambda_1+\lambda_2
 +\lambda_3 \right) ^2\lambda_1\lambda_2\lambda_3-4\, \left( \lambda_1+\lambda_2+\lambda_3 \right)  \left( \lambda_2\lambda_1
 +\lambda_3\lambda_1+\lambda_2\lambda_3 \right) ^2\\
 &+44\, \left( \lambda_1+\lambda_2+\lambda_3 \right)  \left( \lambda_2\lambda_1
+\lambda_3\lambda_1+\lambda_2\lambda_3 \right) \lambda_1\lambda_2\lambda_3-12\, \left( \lambda_2\lambda_1+\lambda_3\lambda_1
+\lambda_2\lambda_3 \right) ^3+12\, \left( \lambda_2\lambda_1+\lambda_3\lambda_1+\lambda_2\lambda_3 \right) ^2\lambda_1\lambda_2
\lambda_3\\
&-12\, \left( \lambda_2\lambda_1+\lambda_3\lambda_1+\lambda_2\lambda_3 \right) \lambda_1^2\lambda_2^2\lambda_3^2-12\, \left( \lambda_1+
\lambda_2+\lambda_3 \right) \lambda_1\lambda_2\lambda_3+4\, \left( \lambda_2\lambda_1+\lambda_3\lambda_1
+\lambda_2\lambda_3 \right) ^2-72\,\lambda_1^2\lambda_2^2\lambda_3^2,
\stepcounter{equation}\tag{\theequation}\label{IgRos}\\
I_6 &= 
-24\, \left( \lambda_1+\lambda_2+\lambda_3 \right) ^3
\lambda_1\lambda_2\lambda_3+10\, \left( \lambda_1+
\lambda_2+\lambda_3 \right) ^2\lambda_1^2\lambda_2^2\lambda_3^2+32\, \left( \lambda_2\lambda_1+
\lambda_3\lambda_1+\lambda_2\lambda_3 \right) ^2
\lambda_1\lambda_2\lambda_3+150\, \left( \lambda_2
\lambda_1+\lambda_3\lambda_1+\lambda_2\lambda_3
 \right) \lambda_1^2\lambda_2^2\lambda_3^2\\
 &+8\,
 \left( \lambda_1+\lambda_2+\lambda_3 \right) ^2 \left( 
\lambda_2\lambda_1+\lambda_3\lambda_1+\lambda_2
\lambda_3 \right) ^2\lambda_1^2\lambda_2^2\lambda_3^2+118\, \left( \lambda_1+\lambda_2+\lambda_3 \right) ^3 \left( \lambda_2\lambda_1+\lambda_3
\lambda_1+\lambda_2\lambda_3 \right) \lambda_1\lambda_2\lambda_3\\
&-194\, \left( \lambda_1+\lambda_2+\lambda_3 \right) ^2 \left( \lambda_2\lambda_1+\lambda_3
\lambda_1+\lambda_2\lambda_3 \right) \lambda_1^2\lambda_2^2\lambda_3^2+118\, \left( \lambda_1+
\lambda_2+\lambda_3 \right)  \left( \lambda_2\lambda_1
+\lambda_3\lambda_1+\lambda_2\lambda_3 \right) ^3
\lambda_1\lambda_2\lambda_3\\
&-66\, \left( \lambda_1+
\lambda_2+\lambda_3 \right)  \left( \lambda_2\lambda_1
+\lambda_3\lambda_1+\lambda_2\lambda_3 \right) ^2\lambda_1^2\lambda_2^2\lambda_3^2+76\, \left( 
\lambda_1+\lambda_2+\lambda_3 \right)  \left( \lambda_2\lambda_1+\lambda_3\lambda_1+\lambda_2\lambda_3
 \right) \lambda_1^3\lambda_2^3\lambda_3^3\\
 &-194\, \left( \lambda_1+\lambda_2+\lambda_3 \right)  \left( 
\lambda_2\lambda_1+\lambda_3\lambda_1+\lambda_2
\lambda_3 \right) ^2\lambda_1\lambda_2\lambda_3+412
\, \left( \lambda_1+\lambda_2+\lambda_3 \right)  \left( 
\lambda_2\lambda_1+\lambda_3\lambda_1+\lambda_2
\lambda_3 \right) \lambda_1^2\lambda_2^2\lambda_3^2\\
&+20\, \left( \lambda_1+\lambda_2+\lambda_3
 \right) ^4 \left( \lambda_2\lambda_1+\lambda_1\lambda_3+\lambda_2\lambda_3 \right) \lambda_1\lambda_2
\lambda_3-36\, \left( \lambda_1+\lambda_2+\lambda_3
 \right) ^3 \left( \lambda_2\lambda_1+\lambda_1\lambda_3+\lambda_2\lambda_3 \right) ^2\lambda_1\lambda_2
\lambda_3\\
&+20\, \left( \lambda_1+\lambda_2+\lambda_3
 \right) ^3 \left( \lambda_2\lambda_1+\lambda_1\lambda_3+\lambda_2\lambda_3 \right) \lambda_1^2\lambda_2^2\lambda_3^2-8\, \left( \lambda_1+\lambda_2+
\lambda_3 \right) ^2 \left( \lambda_2\lambda_1+\lambda_1\lambda_3+\lambda_2\lambda_3 \right) ^3\lambda_1
\lambda_2\lambda_3\\
&+8\, \left( \lambda_1+\lambda_2+\lambda_3 \right) ^2 \left( \lambda_2\lambda_1+\lambda_1\lambda_3+\lambda_2\lambda_3 \right) ^2
-252\,\lambda_1^3\lambda_2^3\lambda_3^3-36\,\lambda_1^4\lambda_2^4\lambda_3^4-24\, \left( \lambda_2\lambda_1+\lambda_1\lambda_3+\lambda_2\lambda_{3
} \right) ^5+48\, \left( \lambda_2\lambda_1+\lambda_1
\lambda_3+\lambda_2\lambda_3 \right) ^4\\
&-24\, \left( \lambda_2\lambda_1+\lambda_1\lambda_3+\lambda_2
\lambda_3 \right) ^3+8\, \left( \lambda_1+\lambda_2+
\lambda_3 \right) ^4 \left( \lambda_2\lambda_1+\lambda_1\lambda_3+\lambda_2\lambda_3 \right) ^2-8\, \left( 
\lambda_1+\lambda_2+\lambda_3 \right) ^3 \left( \lambda_2\lambda_1+\lambda_1\lambda_3+\lambda_2\lambda_3 \right) ^3\\
&+8\, \left( \lambda_1+\lambda_2+\lambda_3
 \right) ^2 \left( \lambda_2\lambda_1+\lambda_1\lambda_3+\lambda_2\lambda_3 \right) ^4 -8\, \left( \lambda_1+
\lambda_2+\lambda_3 \right) ^3 \left( \lambda_2\lambda_1+\lambda_1\lambda_3+\lambda_2\lambda_3 \right) ^2-36\, \left( \lambda_1+\lambda_2+\lambda_3 \right) ^2
 \left( \lambda_2\lambda_1+\lambda_1\lambda_3+\lambda_2\lambda_3 \right) ^3\\
 &+20\, \left( \lambda_1+\lambda_{2
}+\lambda_3 \right)  \left( \lambda_2\lambda_1+\lambda_1\lambda_3+\lambda_2\lambda_3 \right) ^4+20\, \left( 
\lambda_1+\lambda_2+\lambda_3 \right)  \left( \lambda_2\lambda_1+\lambda_1\lambda_3+\lambda_2\lambda_3
 \right) ^3-36\,\lambda_1^2\lambda_2^2\lambda_3^2-24\, \left( \lambda_1+\lambda_2+\lambda_3 \right) ^5\lambda_1\lambda_2\lambda_3\\
 &+48\, \left( \lambda_1+
\lambda_2+\lambda_3 \right) ^4\lambda_1^2\lambda_2^2\lambda_3^2-24\, \left( \lambda_1+\lambda_2+
\lambda_3 \right) ^3\lambda_1^3\lambda_2^3\lambda_3^3+24\, \left( \lambda_1+\lambda_2+\lambda_3 \right) ^4\lambda_1\lambda_2\lambda_3-136\, \left( 
\lambda_1+\lambda_2+\lambda_3 \right) ^3\lambda_1^2\lambda_2^2\lambda_3^2\\
&+32\, \left( \lambda_1+
\lambda_2+\lambda_3 \right) ^2\lambda_1^3\lambda_2^3\lambda_3^3+24\, \left( \lambda_2\lambda_1+
\lambda_1\lambda_3+\lambda_2\lambda_3 \right) ^4 \lambda_1\lambda_2\lambda_3-24\, \left( \lambda_2
\lambda_1+\lambda_1\lambda_3+\lambda_2\lambda_3\right) ^3\lambda_1^2\lambda_2^2\lambda_3^2
+150\, \left( \lambda_1+\lambda_2+\lambda_3 \right) \lambda_1^3\lambda_2^3\lambda_3^3\\
&-136\, \left( 
\lambda_2\lambda_1+\lambda_1\lambda_3+\lambda_2
\lambda_3 \right) ^3\lambda_1\lambda_2\lambda_3+10\,
 \left( \lambda_2\lambda_1+\lambda_1\lambda_3+\lambda_2\lambda_3 \right) ^2\lambda_1^2\lambda_2^2\lambda_3^2-42\, \left( \lambda_2\lambda_1+\lambda_{{1}
}\lambda_3+\lambda_2\lambda_3 \right) \lambda_1^3\lambda_2^3\lambda_3^3\\
&-42\, \left( \lambda_1+
\lambda_2+\lambda_3 \right) \lambda_1^2\lambda_2^2\lambda_3^2+76\, \left( \lambda_1+\lambda_2+
\lambda_3 \right)  \left( \lambda_2\lambda_1+\lambda_1\lambda_3+\lambda_2\lambda_3 \right) \lambda_1\lambda_2
\lambda_3-66\, \left( \lambda_1+\lambda_2+\lambda_3 \right) ^2 \left( \lambda_2\lambda_1+\lambda_1\lambda_3
+\lambda_2\lambda_3 \right) \lambda_1\lambda_2\lambda_3, \\
I_{10} & =\lambda_1^2\lambda_2^2\lambda_3^2
 \left( \lambda_3-1 \right) ^2 \left( \lambda_2-1 \right) ^2 \left( -\lambda_3+\lambda_2 \right) ^2 \left( \lambda_1-1 \right) ^2 \left( -\lambda_3+\lambda_1 \right) ^2
 \left( -\lambda_2+\lambda_1 \right) ^2.
\end{align*}
\endgroup

The components of the rational map $\Phi: \M_2 \backslash \supp{ (\chi_{35})}_0 \to \M_2$ with $(j_1, j_2, j_3) \mapsto (j'_1, j'_2, j'_3)$ are given by
\begin{equation*}
\begin{split}
 j_1^\prime 	= \frac{64}{729} 	\, \frac{g^{(1)}(j_1,j_2,j_3)}{q(j_1,j_2,j_3)},\quad
 j_2^\prime 	= \frac{4}{729} 	\, \frac{g^{(2)}(j_1,j_2,j_3)}{q(j_1,j_2,j_3)}, \quad
 j_3^\prime 	= \frac{1}{729} 	\, \frac{g^{(3)}(j_1,j_2,j_3)}{q(j_1,j_2,j_3)},
\end{split}
\end{equation*} 
with
\begingroup\makeatletter\def\f@size{7}\check@mathfonts\def\maketag@@@#1{\hbox{\m@th\large\normalfont#1}}%
 \begin{align*}
q(j_1,j_2,j_3) & =  j_1^5 \, \Big(j_2^4j_1^3-12\,j_1^3j_2^3j_3+54\,j_1^3j_2^2j_3^2-108\,j_1^3j_2j_3^3+81\,j_1^3j_3^4
 +78\,j_2^5j_1^2-1332\,j_1^2j_2^4j_3+8910\,j_1^2j_2^3j_3^2-29376\,j_1^2j_2^2j_3^3+47952\,j_1^2j_2j_3^4\\
 &-31104\,j_1^2j_3^5-159\,j_1j_2^6+1728\,j_1j_2^5j_3-6048\,j_1j_2^4j_3^2+6912\,j_1j_2^3j_3^3+80\,j_2^{7}-384\,j_2^6j_3
 -972\,j_1^4j_2^2+5832\,j_1^4j_2j_3-8748\,j_1^4j_3^2-77436\,j_1^3j_2^3\\
 &+870912\,j_1^3j_2^2j_3-3090960\,j_1^3j_2j_3^2 +3499200\,j_1^3j_3^3+592272\,j_2^4j_1^2-4743360\,j_1^2j_2^3j_3+9331200\,j_1^2j_2^2j_3^2-41472\,j_1j_2^5
+236196\,j_1^5\\
&+19245600\,j_2j_1^4-104976000\,j_1^4j_3-507384000\,j_2^2j_1^3+2099520000\,j_1^3j_2j_3+125971200000\,j_1^4\Big)  ,\\
 g^{(1)}(j_1,j_2,j_3) & = \Big(- j_2^2j_1+6\,j_2j_3j_1-9\,j_3^2j_1+j_2^3+540\,j_1^2 \Big)^5 ,\\
g^{(2)}(j_1,j_2,j_3) & = \Big( j_2^4j_1^2-12\,j_1^2j_2^3j_3+54\,j_1^2j_2^2j_3^2-108\,j_1^2j_2j_3^3+81\,j_1^2j_3^4
-2\,j_1j_2^5+12\,j_1j_2^4j_3-18\,j_1j_2^3j_3^2+j_2^6-756\,j_2^2j_1^3+4536\,j_1^3j_2j_3-6804\,j_1^3j_3^2\\
&+5130\,j_1^2j_2^3-17496\,j_1^2j_2^2j_3+131220\,j_1^4-2332800\,j_2j_1^3 \Big) \Big( -j_2^2j_1+6\,j_2j_3j_1-9\,j_3^2j_1+j_2^3+540\,j_1^2 \Big)^3 , 
 \stepcounter{equation}\tag{\theequation}\label{Components_Phi}\\
 g^{(3)}(j_1,j_2,j_3) & = \Big( -j_1^3j_2^6+18\,j_1^3j_2^5j_3-135\,j_1^3j_2^4j_3^2+540\,j_1^3j_2^3j_3^3
 -1215\,j_1^3j_2^2j_3^4+1458\,j_1^3j_2j_3^5-729\,j_1^3j_3^6+3\,j_1^2j_2^{7}-36\,j_1^2j_2^6j_3
  +162\,j_1^2j_2^5j_3^2-324\,j_1^2j_2^4j_3^3\\
 & +243\,j_1^2j_2^3j_3^4-3\,j_1j_2^{8}+18\,j_1j_2^{7}j_3
  -27\,j_1j_2^6j_3^2+j_2^{9}+1350\,j_1^4j_2^4-16200\,j_1^4j_2^3j_3+72900\,j_1^4j_2^2j_3^2-145800\,j_1^4j_2j_3^3
  +109350\,j_1^4j_3^4-6345\,j_1^3j_2^5\\
 & +52650\,j_1^3j_2^4j_3-144585\,j_1^3j_2^3j_3^2+131220\,j_1^3j_2^2j_3^3
  +4995\,j_1^2j_2^6-14580\,j_1^2j_2^5j_3-599724\,j_1^5j_2^2+3598344\,j_1^5j_2j_3-5397516\,j_1^5j_3^2
  +4175226\,j_1^4j_2^3\\
  & -15390648\,j_1^4j_2^2j_3+4898880\,j_1^4j_2j_3^2-1961496\,j_1^3j_2^4+87392520\,j_1^6
  -881798400\,j_1^5j_2-1259712000\,j_1^5j_3 \Big) \\
  &\times \Big( -j_1j_2^2+6\,j_1j_2j_3-9\,j_1j_3^2+j_2^3+540\,j_1^2 \Big)^2 .\\
 \end{align*}
\endgroup
\end{landscape}

\bibliographystyle{amsplain}

\bibliography{ref}{}

\end{document}